\documentclass[11pt, a4paper, titlepage]{article}
\usepackage[english]{babel}
\usepackage[T1]{fontenc}
\usepackage[utf8]{inputenc}

\usepackage{textgreek}
\usepackage{sidecap}
\usepackage{titlesec}
\usepackage[usenames,dvipsnames,svgnames,table]{xcolor}
\usepackage{graphicx}
\usepackage{caption}
\usepackage{subcaption}
\usepackage{siunitx}
\usepackage{xcolor}
\usepackage{enumitem}

\usepackage{graphicx}
\usepackage{rotating}

\usepackage{caption}
\usepackage{amsmath}
\usepackage{amsfonts}
\usepackage{amsthm}
\usepackage{mathtools}
\usepackage{amssymb}
\usepackage{mathrsfs}
\usepackage{dsfont}

\allowdisplaybreaks

\newtheoremstyle{mythm}%
{3pt}
{3pt}
{\itshape\color{black}}
{}
{\bfseries\color{blue}}
{.}
{.5em}
{}

\newtheoremstyle{mydef}%
{3pt}
{3pt}
{\upshape\color{black}}
{}
{\bfseries\color{blue}}
{.}
{.5em}
{}

\theoremstyle{mythm}

\newtheorem{theorem}{Theorem}[section] 
\newtheorem{corollary}[theorem]{Corollary}
\newtheorem{proposition}[theorem]{Proposition}
\newtheorem{lemma}[theorem]{Lemma}
\newtheorem*{theorem*}{Theorem}

\theoremstyle{mydef}
\newtheorem{definition}[theorem]{Definition}

\numberwithin{equation}{section}

\usepackage{pgf,tikz}
\usetikzlibrary{calc,patterns,angles,quotes}

\usepackage{float}
\usetikzlibrary{shapes.misc}
\tikzstyle{root}=[circle,fill=black,inner sep=0pt,minimum size=8pt]
\tikzstyle{steiner}=[circle,fill=blue,inner sep=0pt,minimum size=8pt]
\tikzstyle{terminal}=[fill=red]

\tikzstyle{X}=[circle,fill=blue,inner sep=0pt,minimum size=6pt]
\tikzstyle{T}=[fill=red,inner sep=0pt,minimum size=4pt]
\tikzstyle{C}=[fill=red,inner sep=0pt,minimum size=4pt]
\tikzstyle{R}=[circle,fill=black,inner sep=0pt,minimum size=4pt]
\tikzstyle{background-line}=[color=green!55]

\tikzset{cross/.style={cross out, draw=black, minimum size=2*(#1-\pgflinewidth), inner sep=0pt, outer sep=0pt},
	cross/.default={2pt}}
	
\usepackage[dvipsnames]{xcolor}

\definecolor{mygray}{gray}{0.28}
\definecolor{myorange}{RGB}{220,120,50}

\newcommand{\oc}{i_c}
\newcommand{\ipinf}{i_{+\infty}}
\newcommand{\iminf}{i_{-\infty}}
\newcommand{\ta}{x}

\newcommand{\nl}{\nonumber \\}

\usepackage{nccfoots}
\usepackage{scalerel}
\usetikzlibrary{svg.path}

\definecolor{orcidlogocol}{HTML}{A6CE39}
\tikzset{
  orcidlogo/.pic={
    \fill[orcidlogocol] svg{M256,128c0,70.7-57.3,128-128,128C57.3,256,0,198.7,0,128C0,57.3,57.3,0,128,0C198.7,0,256,57.3,256,128z};
    \fill[white] svg{M86.3,186.2H70.9V79.1h15.4v48.4V186.2z}
                 svg{M108.9,79.1h41.6c39.6,0,57,28.3,57,53.6c0,27.5-21.5,53.6-56.8,53.6h-41.8V79.1z M124.3,172.4h24.5c34.9,0,42.9-26.5,42.9-39.7c0-21.5-13.7-39.7-43.7-39.7h-23.7V172.4z}
                 svg{M88.7,56.8c0,5.5-4.5,10.1-10.1,10.1c-5.6,0-10.1-4.6-10.1-10.1c0-5.6,4.5-10.1,10.1-10.1C84.2,46.7,88.7,51.3,88.7,56.8z};
  }
}

\newcommand\orcidicon[1]{\href{https://orcid.org/#1}{\mbox{\scalerel*{
\begin{tikzpicture}[yscale=-1,transform shape]
\pic{orcidlogo};
\end{tikzpicture}
}{|}}}}

\usepackage{hyperref}

\begin{document}

\thispagestyle{plain}
\begin{center}
    \huge{\textbf{Convective stability of the critical waves of an FKPP growth process}}
        
    \vspace{0.8cm}
    \Large{Florian Kreten}* \orcidicon{0000-0003-1938-2590} \\
    \vspace{0.2cm}
    \Large{\today}
       
    \vspace{1cm}
    \textbf{Abstract}
\end{center}
We construct the traveling wave solutions of an FKPP growth process of two densities of particles, and prove that the critical traveling waves are locally stable in a space where the perturbations can grow exponentially at the back of the wave. The considered reaction-diffusion system was introduced by Hannezo et al. in the context of branching morphogenesis (Cell, 171(1):242-255.e27, 2017): active, branching particles accumulate inactive particles, which do not react. Thus, the system features a continuum of steady state solutions, complicating the analysis. We adopt a result by Faye and Holzer (J.Diff.Eq., 269(9):6559-6601, 2020) for proving the stability of the critical traveling waves, by modifying the semi-group estimates to a space with unbounded weights. The novelty is that we use a Feynman-Kac formula to get an exponential a-priori estimate for the left tail of the PDE, in the regime where the weight is unbounded.\\
\vspace{0.2cm}

\noindent \textbf{Key words}: Traveling waves, Local stability, Convective stability, Reaction-Diffusion equation, Branching particle system.\\
\textbf{MSC2020}: 35B35, 35C07, 35K57, 34E10, 92C15.\\
\vspace{0.2cm}

\footnotesize{* Institut für Angewandte Mathematik, Rheinische Friedrich-Wilhelms-Universität, Endenicher Allee 60, 53115 Bonn, Germany. Email: florian.kreten@uni-bonn.de.\\
This work was partly funded by the Deutsche Forschungsgemeinschaft (DFG, German Research Foundation) under Germany’s Excellence Strategy - GZ 2047/1, Projekt-ID 390685813 and by the Deutsche Forschungsgemeinschaft (DFG, German Research Foundation) - Projektnummer 211504053 - SFB 1060.\\
\textbf{Acknowledgements:} The author would like to thank Anton Bovier and Juan Velázquez for the discussions and their support.}
\newpage

\normalsize

\section{Introduction and results}

We analyze an FKPP-system \cite{Fisher_1937_Wave, KPP_1937_Wave} that models a self-organized growth process. Considering the one-dimensional case $z \in \mathds{R}, t \in \mathds{R}^+_0$, the densities $A(t,z), I(t,z) \geq 0$ of active and inactive particles follow dynamics
\begin{align}
\begin{aligned} \label{Eq:Perturbed_PDE}
\frac{\partial}{\partial t} A &= \frac{\partial^2}{\partial z^2}A + A - A(A+I), \\
\frac{\partial}{\partial t} I &= d \frac{\partial^2}{\partial z^2} I + rA + A(A+I), \qquad r,d \geq 0.
\end{aligned}
\end{align}
This system was introduced by Hannezo et al. in the context of branching morphogenesis \cite{Hannezo_2017_Unifying}. The authors used a stochastic branching particle system to model the morphogenesis of branched glandular structures. The PDE \eqref{Eq:Perturbed_PDE} for $d=0$ is the heuristic hydrodynamic limit of their stochastic system. Existence and uniqueness of non-negative solutions of \eqref{Eq:Perturbed_PDE} follow by classical fixed-point theory, see e.g. chapter 14 in \cite{Smoller_Shocks_Reaction_Diffusion}. Given the normalized System \eqref{Eq:Perturbed_PDE}, the general case can be obtained by rescaling \cite{Kreten2022}.

We construct the traveling waves of System \eqref{Eq:Perturbed_PDE} and prove that for $d>0$, those with minimal speed are locally stable against perturbations. The difficulty when analyzing this system is that the inactive particles $I$ do not react, only the active particles $A$ branch and become inactive upon collision. Thus, the system features a continuum of steady states
\begin{align}
P_I = \{A = 0, I = K \, \vert \, K \in \mathds{R}\}, \label{Eq:Cont_deg_fixed_points}
\end{align}
and a-priori, we do not know which of these steady states are relevant.

\begin{figure}[h]
 		\centering
 		\begin{minipage}[c]{0.90\textwidth}
\begin{picture}(100,100)
	\put(0,0){\includegraphics[width=\textwidth]{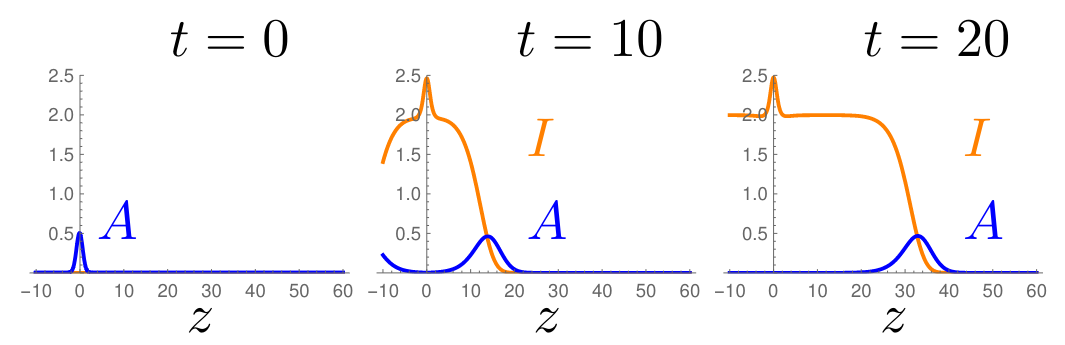}}
\end{picture} 	
 		\end{minipage}
 	\begin{minipage}[c]{\textwidth}
 	\vspace{0.2cm}
	\caption{Simulation of System \eqref{Eq:Perturbed_PDE} for $r, d = 0$. Given an initial heap of active particles $A(z,0)=1/2 \exp (-z^2)$, and $I(z,0) = 0$, two identical traveling fronts arise, the right one is shown. After the separation of the two fronts away from the origin, the density of the remaining inactive particles is given by $I = 2$ and the front moves asymptotically with speed $c=2$.} \label{Diff_system_pics}
	\end{minipage}
\end{figure}

Simulations show that the system forms traveling wave solutions, which select for particular steady states among all possible ones, see Figure \ref{Diff_system_pics}. FKPP-systems are well-known to form heteroclinic traveling wave solutions, that connect two different steady states \cite{Bramson_1983ConvergenceOS, Fisher_1937_Wave, KPP_1937_Wave, FAYE_lotke_volterra_critical_stability, Ducrot2019_Consecutive_Fronts}. A (right-) traveling wave solution is constant in the moving frame $x = z -ct$, for a wave-speed $c >0 $. Thereby, we refer to a traveling wave as a non-constant and bounded solution of the system of ODEs
\begin{align}
\begin{aligned}
0 &= c \frac{\partial}{\partial x} a + \frac{\partial^2}{\partial x^2} a + a -a(a+i), \\
0 &= c \frac{\partial}{\partial x} i + d \frac{\partial^2}{\partial x^2}i + ra + a(a+i). \label{Eq:perturbed_wave}
\end{aligned}
\end{align}
We call a solution of System \eqref{Eq:perturbed_wave} non-negative if $a,i \geq 0$. Moreover, an \textit{invading front} is a non-negative traveling wave where both $a(x)$ and $i(x)$ vanish as $x \rightarrow + \infty$.

For the case $d=0$, we constructed the traveling waves of the system, but could not analyze their stability \cite{Kreten2022}. Since the inactive particles neither react nor diffuse, any deviation from the traveling wave remains in the system for all times (see Fig. \ref{Diff_system_pics}). Therefore, we introduce the diffusion to the inactive particles. The present article is divided into two parts:

\textit{1)} In Section \ref{Sec:Existence_big_section}, given the traveling waves for $d=0$ as our starting point, we apply perturbation techniques to construct the traveling waves for $d \neq 0$. These waves are continuous in $d \geq 0$, so we recover the original dynamics as $d \rightarrow 0$. We prove the existence of a continuum of traveling wave solutions, that correspond to the continuum of steady states \eqref{Eq:Cont_deg_fixed_points}, check Theorems \ref{Old_main_Theorem} and \ref{Prop:Existence_of_a_tr_sol}. For this introduction, we restrict to the invading fronts. In particular, there exists an invading front with minimal possible speed $c=2$, referred to as \textit{critical front}:

\begin{theorem} \label{Thm:existence_small_epsilon}
For $r \geq 0, c>0$, consider the Wave System \eqref{Eq:perturbed_wave} with
\begin{align}
0< d  < \min \big \{ 1, \frac{3c}{2} , \frac{c^2}{2(r+1)} \big \}. \label{Eq:Contidion_existence_final}
\end{align}
If and only if $c \geq 2$, there exists an invading front. The function $i(x)$ is decreasing, and $a(x)$ has a unique local and global maximum. As $x \rightarrow - \infty$, the front converges exponentially fast to a fixed point $(a,i) = (0,\iminf)$, where
\begin{align}
1< 2 - d \cdot \frac{ 2(r+1) }{c} < \iminf < 2. \label{Eq:bound_invading_iminf}
\end{align}
The rate of convergence is a function of  $\iminf$, given by
\begin{align}
\mu_{-\infty} = - \frac{c}{2} + \sqrt{ \frac{c^2}{4} + i_{- \infty} -1} >0.
\end{align}
There are two possibilities for the speed of convergence as $x \rightarrow + \infty$. In the critical case $c = 2$, the front behaves as $x \cdot e^{-x}$. If $c>2$, convergence is purely exponential, with rate
\begin{align}
\mu_{+ \infty} =  \frac{c}{2} - \sqrt{ \frac{c^2}{4} -1} >0. \label{Eq:Rates_Conv_Thm}
\end{align}
\end{theorem}
Such an invading front arises in the simulation depicted in Figure \ref{Diff_system_pics}: a pulse of active particles is accompanied by a monotone wave of inactive particles.\\

\textit{2)} In Section \ref{Sec:Stability_big}, we analyze the stability of the critical invading front. In the moving frame $x = z -2t$, we write a solution of the PDE \eqref{Eq:Perturbed_PDE} as the sum of the front $a(x), i(x)$ and the perturbations $\tilde{A}(t,x), \tilde{I}(t,x)$:
\begin{align}
A(t,x ) = a(x) + \tilde{A}(t,x), \quad I(t,x) = i(x) + \tilde{I}(t,x).
\end{align}

The system is not attracted towards a particular limit, as none of the Steady States \eqref{Eq:Cont_deg_fixed_points} is stable in the classical sense. To overcome this, we operate in a space where we allow the perturbations to grow exponentially as $x \rightarrow - \infty$. This type of stability is referred to as \textit{convective stability} \cite{Ghazaryan_overview}, since we more or less ignore what happens to perturbations that are convected to the back of the invading front. To stabilize the front as $x \rightarrow + \infty$, perturbations must vanish exponentially fast, which is typical for FKPP-fronts \cite{SATTINGER_stability, SANDSTEDE_stability_traveling}. Given a smooth weight $w(x) >0$, subject to
\begin{align}
w(x) = \begin{cases}
e^{- x} & \quad x \geq 1, \\
e^{ -\alpha x} & \quad x \leq -1, \label{Def:Weight1} \quad \text{with fixed } \alpha \in (0,1), \\
1  & \quad x = 0,
\end{cases}
\end{align}
we prove that if the weighted perturbations
\begin{align}
\frac{\tilde{A}(t,x)}{w(x)}, \quad \frac{ \tilde{I}(t,x)}{w(x)}
\end{align}
are initially small, they vanish point-wise with algebraic decay $t^{-3/2}$. In order to deal with the unbounded weight, we need some a-priori control of the left tail of the PDE. The novelty of the present work is that we use a Feynman-Kac formula to prove exponential decay of the left tail of the perturbations, under the assumption that also the unweighted perturbations are initially small.

We have to assume that the discrete spectrum of the linearized perturbation equation contains no elements with non-negative real-part, check Section \ref{Sec:Prelim_Discussion_Space} for the precise statement. We verify this numerically, the technique is presented in Appendix \ref{Sec:Numerical_spectrum}, an analytic proof remains an open problem.

\begin{theorem} \label{Thm:large_time_decay}
For a pair $ d >0, r \geq 0$ as in Theorem \ref{Thm:existence_small_epsilon}, consider the critical invading front with speed $c=2$. If we assume that the discrete spectrum of the weighted linearized perturbation equation contains no elements with non-negative real-part (see \eqref{Ass:Point_spectrum}), then the critical invading front is locally stable in a space with weight $w(x)$ as above \eqref{Def:Weight1}:

Fix a pair of constants $C, \mu_0 >0$. For all $\epsilon >0$, there exists a $\delta >0$, such that if the unweighted perturbations fulfill
\begin{align}
& (i) & \forall x \leq 0: \quad  \vert \tilde{A}(0,x)  \vert  \leq C e^{\mu_0 x}, \quad \vert \tilde{I}(0,x) & \vert  \leq \delta, \label{Unweighted_Requirements}
\intertext{and if the weighted perturbations $\tilde{A}(0,x)/w(x), \tilde{I}(0,x) / w(x)$ are elements of $H^2(\mathds{R})$, and}
& (ii) & \int_{ \mathds{R}}(1+ \vert x \vert) \Big [  \big \vert \frac{\tilde{A}(0,x)}{w(x)} \big \vert + \big \vert \frac{ \tilde{I}(0,x)}{w(x)} \big \vert \Big ]\, dx & \leq \delta, \\
& (iii) &  \big \vert \big \vert \frac{\tilde{A}(0,x)}{w(x)} \big \vert \big \vert_{\infty} +  \big \vert \big \vert \frac{ \tilde{I}(0,x)}{w(x)} \big \vert \big \vert_{\infty} & \leq \delta,
\end{align}
then the weighted perturbations decay point-wise with algebraic speed $t^{-3/2}$:
\begin{align}
\sup_{t \geq 0} \sup_{x \in \mathds{R}} \, \frac{(1+t)^{3/2}}{w(x)  (1+ \vert x\vert) } \Big (  \vert \tilde{A}(t,x) \vert  +  \vert \tilde{I}(t,x) \vert  \Big ) \leq \epsilon.
\end{align}
\end{theorem}

\subsection{Background}

One of the central questions regarding tissue growth is how individual cells react to their microscopic environment, and how this gives rise to distinct macroscopic structures. Mathematical models help to understand these processes, there exists a huge literature of works regarding organ formation, wound healing or tumor growth \cite{DAlessandro2021_Cell_Memory_Migration, Falco2021_Clinical_Cancer_Modeling_Glioblastoma, Ladoux_Cell_Migration_Plasticity,
Leroy_Tumor_Mechanical_2017, Mammoto2010_Organ_Development}. While most of the clinically relevant results are numerical, simplified reaction-diffusion systems provide a framework for rigorously analyzing the arising spatiotemporal patterns. The first study of this kind was the seminal analysis of the FKPP-equation in the 1930s \cite{Fisher_1937_Wave, KPP_1937_Wave}, describing the spreading of a fitter population, or "The Wave of Advance of Advantageous Genes" \cite{Fisher_1937_Wave}. Given System \eqref{Eq:Perturbed_PDE}, one obtains the original FKPP-equation when substituting $I \equiv 0$.

Hannezo et al. modeled the branching morphogenesi of glandular structures in organs via a stochastic system that is based on branching and annihilating random walks \cite{Hannezo_2017_Unifying, Cardy_Tauber_BARW}. Their numerical results suggest that this system self-organizes into spatially homogeneous structures, despite the fact that it is based on local mechanisms alone (i.e. there are no global guiding gradients). The authors also suggested the PDE \eqref{Eq:Perturbed_PDE} with $d=0$ to study the mean-field behavior of the stochastic system. Since all mechanisms in System \eqref{Eq:Perturbed_PDE} are purely local, it can be seen as a degenerate Keller-Segel system \cite{Keller_bacteria, Painter_Keller_Segel_SelfOrganization_2019, Perthame_Keller_Segel_2004}. For such systems, various organization phenomena are known, and can be proven in some cases \cite{Arumugam_2020_Keller-Segel}.

Simulations indicate that System \eqref{Eq:Perturbed_PDE} forms traveling solutions for a wide range of initial data, an example is depicted in Figure \ref{Diff_system_pics}. However, a global convergence result seems difficult to prove since the active particles form a non-monotone pulse, being accompanied by a monotone front of inactive particles of the very same speed. Thereby, the system evades any classical comparison principles. For a Lokta-Volterra system where each consecutive species spreads with a different speed, Ducrot et al. \cite{Ducrot2019_Consecutive_Fronts} could prove a global convergence result: since the fronts separate, each single one satisfies a comparison principle.

For systems without a comparison principle, typically only the stability against small perturbations can be proven. For noncritical fronts (those with a spectral gap of the linearized perturbation) and for bounded weights, Sattinger \cite{SATTINGER_stability} proved a general result in 1976. More recently, Ghazaryan et al. \cite{Ghazaryan_overview} proved several general results in the case of unbounded weights, even for mixed ODE-PDE models, but also under the assumption of a spectral gap. They also prefer the notion of \textit{convective stability} \cite{Ghazaryan_overview}, to emphasize the point-wise stability (whereas some physicists or biologists prefer to call the same phenomenon \textit{convective instability}, emphasizing the fact that the overall perturbation does not vanish or might even increase \cite{Sherratt_Convective_Instability}). For an a-priori control of the unbounded nonlinearity, Ghazaryan et al. use an energy estimate in $H^2$, which is not compatible with our point-wise result \cite{Ghazaryan_Convective_First_Work}.

The critical case - where the essential spectrum is only marginally stable - requires a different approach. Only for the case of a single reactant, the stability of critical FKPP-fronts has been proven for quite general reaction-terms, by Gallay in 1994 \cite{Gallay_1994_stability_critical}. The author proved that the algebraic speed of decay $t^{-3/2}$ is indeed optimal, we expect this also to be true in the present case. For more than one reactant, rigorous proofs are sparse.

Fay\'{e} and Holzer \cite{FAYE_lotke_volterra_critical_stability} could prove the local stability of the critical waves of a Lokta-Volterra system of two species. They generalize a technique for dealing with marginally stable spectra, which they introduced in 2018 for a system with a single reactant \cite{Faye2018_critical_FKPP}, to systems of two reactants. The underlying technique goes back to Zumbrun and Howard \cite{Zumbrun_Howard_Pointwise_Stability}, in principle it should be possible to generalize this result to systems of $n >2 $ reactants, but the notation becomes quite tedious. The computations are quite demanding: a precise analysis of the Evans-function in the neighborhood of the essential spectrum is required, making use of the Gap Lemma that goes back to Gardner and Zumbrun \cite{Gardner_Zumbrun_Gap_Lemma} - involving several changes of coordinates while keeping track of the resulting error terms. We are in the lucky situation that the linearized perturbation falls into the marginally stable category analyzed in \cite{FAYE_lotke_volterra_critical_stability}, at least if we use an unbounded weight as in Theorem \ref{Thm:large_time_decay}. I want to thank the authors for remarking that their proof is largely independent of the actual weight, which initiated the idea for this project.

We adapt the resulting semi-group estimates to the case where the weight is not integrable. This is why we need an a-priori bound on the tail of the PDE \eqref{Eq:Perturbed_PDE} in the moving frame. We can then estimate the decay of a linear super-solution via a Feynman-Kac formula, where we can stop the driving Brownian motion at an arbitrarily chosen point, e.g. $x=0$. As long as the perturbation at this specific point remains small, the super-solution decays exponentially. Our approach of dealing with the unbounded weight can be adapted to other systems with non-negative solutions, given that they are asymptotically monotone and that the nonlinearity essentially depends on a single reactant (which is the reason why we need the unbounded weight at all), check the introduction to Section \ref{Sec:Stability_big}.

\subsection{Structure of the paper}

In Section \ref{Sec:Stability_big}, we prove the local stability of the critical invading front. After introducing the necessary objects in Section \ref{Sec:Notation_Pert}, we present the central steps of the proof in Section \ref{Sec:Central_steps_estimate}. The technical details are then given in Section \ref{Sec:Prelim_Discussion_Space} and Section \ref{Sec:Long_time_estimates}. The Feynman-Kac formula, which controls the left tail of the PDE, is proven in Appendix \ref{Sec:A_priori_estimates}. In Appendix \ref{Sec:Numerical_spectrum}, we present a numerical evaluation of the non-negative discrete spectrum of the linearized perturbation.

In Section \ref{Sec:Existence_big_section}, we construct the traveling waves. Given the traveling waves for $d=0$ (Thm. \ref{Old_main_Theorem}), we apply perturbation techniques to track any finite segment of the waves for $d \neq 0$, see Section \ref{Sec:Asymptotics}. The singular perturbation (for passing continuously from $d=0$ to $d \sim 0$) is explained in Appendix \ref{Sec:Geom_sing_appl}. In Section \ref{Sec:Properties}, we analyze the phase space of the non-negative waves. We then extend the previous perturbation result and prove that a traveling wave persists locally under perturbation in $d$, up to its limit as $x \rightarrow + \infty$. Given any traveling wave, we prove an estimate of type $\iminf + \ipinf = 2 + \mathcal{O}(d)$ regarding the limits in Section \ref{Sec:Mapping_Limits}. We then prove the existence of non-negative traveling waves and invading fronts up to $d \sim 1$, in Section \ref{Sec:Existence_invading_front}.


\section{The stability of the critical invading front}
\label{Sec:Stability_big}

\subsection{Notation}
\label{Sec:Notation_Pert}

In the following, we assume without further mentioning that $a(x), i(x)$ is a non-negative critical invading front as in Theorem \ref{Thm:existence_small_epsilon}. To begin with, we decompose any solution of the PDE \eqref{Eq:Perturbed_PDE} in the moving frame $x = z -ct$ into
\begin{align}
A(t,x ) = a(x) + \tilde{A}(t,x), \quad I(t,x) = i(x) + \tilde{I}(t,x).
\end{align}
Then, the perturbations $\tilde{A}(t,x)$ and $\tilde{I}(t,x)$ follow
\begin{align}
\begin{aligned}
\frac{\partial}{\partial t} \tilde{A} & = \frac{\partial^2}{\partial x^2} \tilde{A} + c \frac{\partial}{\partial x}\tilde{A} + \tilde{A}(1-2a-i) - \tilde{I}a - \tilde{A}(\tilde{A} + \tilde{I}),
	\\
\frac{\partial}{\partial t}\tilde{I} & = d \frac{\partial^2}{\partial x^2} \tilde{I} + c \frac{\partial}{\partial x} \tilde{I} + \tilde{A}(2a+i+r) + \tilde{I}a +\tilde{A}(\tilde{A} + \tilde{I}). \label{Eq:Non_Linear_PDE}
\end{aligned}
\end{align}

For $\tilde{A},\tilde{I}$ that solve the perturbation Eq. \eqref{Eq:Non_Linear_PDE} and given a strictly positive weight $w(x)$, we define the weighted perturbations $u: = \tilde{A} / w, v: = \tilde{I} /w$. If $w$ is twice differentiable with derivatives $w',w''$, they solve
\begin{align}
\begin{aligned}
\frac{\partial}{\partial t} u &=  \frac{\partial^2}{\partial x^2} u + \frac{\partial}{\partial x} u \cdot \big ( c + 2\frac{w'}{w} \big ) + u \cdot \big (
c \frac{w'}{w} + \frac{ w''}{w} \big ) \\
& \qquad \qquad + u(  1-2a-i  ) - va -wu(u+v), \\
\frac{\partial}{\partial t} v &= d  \frac{\partial^2}{\partial x^2} v + \frac{\partial}{\partial x} v \cdot \big ( c + 2d\frac{w'}{w} \big ) + v \cdot \big (
c \frac{w'}{w} + d \frac{ w''}{w} \big ) \\
& \qquad \qquad + u(  2a+i+r) +va +wu(u+v). \label{Eq:Perturbed_Oparator_corrected}
\end{aligned}
\end{align}
We summarize this as
\begin{align}
\frac{\partial}{ \partial t}(u,v) &= \mathcal{L} (u,v) + N(u,v),
\intertext{with the linear part}
\mathcal{L}u &:= \frac{\partial^2}{\partial x^2} u + \frac{\partial}{\partial x} u \cdot \big ( c + 2\frac{w'}{w} \big ) + u \cdot \big (
c \frac{w'}{w} + \frac{ w''}{w} \big )
	\nonumber \\
 & \hspace{0.5cm} + u(  1-2a-i  ) -va,
	\nonumber \\
\mathcal{L}v &:= d \frac{\partial^2}{\partial x^2} v + \frac{\partial}{\partial x} v \cdot \big ( c + 2d\frac{w'}{w} \big ) + v \cdot \big (
c \frac{w'}{w} + d \frac{ w''}{w} \big )
		\label{Eq:Linear_Operator} \\
& \hspace{0.5cm} + u(  2a+i+r) +va, \nonumber
\intertext{and the nonlinear part}
N u &:= -wu(u+v), \qquad
N v := wu(u+v).
\end{align}
For compactness of presentation, we introduce the vectorial notation
\begin{align}
p(t,x) := \begin{pmatrix}
u(t,x) \\
v(t,x)
\end{pmatrix}, \,
N(p)(t,x) := w(x)u(t,x) \cdot
\begin{pmatrix}
- u(t,x) - v(t,x) \\
 u(t,x) + v(t,x) 
\end{pmatrix}. \label{Def:Notation_p}
\end{align}
Lastly, we define as $\mathcal{G}(t,x,y)$ the Kernel of the linear Eq. \eqref{Eq:Linear_Operator}:
\begin{align}
\mathcal{G}(t,x,y) := \begin{pmatrix}
\mathcal{G}_{11}(t,x,y) & \mathcal{G}_{12}(t,x,y)  \\ \mathcal{G}_{21}(t,x,y) & \mathcal{G}_{22}(t,x,y)
\end{pmatrix},
\end{align}
acting on a space to be defined later. In this compact notation, we will estimate the evolution of the weighted perturbations with a Duhamel principle:
\begin{align}
p(t,x) = \int_{ \mathds{R}} \mathcal{G}(t,x,y) p(0,y) \, dy +
\int_0^t ds \int_\mathds{R} \mathcal{G}(t-s,x,y) N (p)(s,y) \, dy. \label{Eq:Duhamels_formula}
\end{align}
Eq. \eqref{Eq:Duhamels_formula} holds if the nonlinearity is locally Lipschitz, see e.g. Pazy p. 185 ff \cite{Pazy_Operators_and_PDE}.

\subsection{Central steps}
\label{Sec:Central_steps_estimate}

For $d>0$ and in a function space where the weight grows exponentially as $x \rightarrow - \infty$, the operator $\mathcal{L}$ is sectorial \cite{Henry_Geometric_Parabolic, Alexander_Gardner_Evans_Invariance, Gardner_Zumbrun_Gap_Lemma}, see Section \ref{Sec:Prelim_Discussion_Space} for an analysis of its spectrum: the spectrum is contained in the strictly negative half-plane, with the exception of a single half-line that touches the origin. In this setting, we can use a result of Fay\'{e} and Holzer \cite{FAYE_lotke_volterra_critical_stability} to estimate the long-time behavior of $\mathcal{G}$, check Section \ref{Sec:Long_time_estimates}. Roughly, this result states that $\mathcal{G}$ decays like $t^{-3/2}$ for large $t$. Then, for estimating the evolution of the full nonlinear system via Eq. \eqref{Eq:Duhamels_formula}, it is crucial to control the nonlinear integral
\begin{align}
 \int_{\mathds{R}} w (x) u(t,x) \cdot \big ( u(t,x) + v(t,x) \big ) \, dx . \label{Eq:Integral_needed}
\end{align}

We treat the cases $x \leq 0$ and $x \geq 0$ separately. Assuming that $w(x)$ vanishes exponentially fast as $z \rightarrow + \infty$, the front satisfies the classical estimate
\begin{align}
\begin{aligned}
& \Big \vert \int_{0}^\infty w (x) u(t,x) \cdot \big ( u(t,x) + v(t,x) \big ) \, dx \Big \vert
\\
& \leq \sup_{x \geq 0} \Big \{ \big (  \vert u(t,x) \vert  +  \vert v(t,x) \vert  \big)^2 \Big \} \cdot \int_{0}^\infty w (x) \, dx
\\
& \leq C \cdot \sup_{x \geq 0} \big \{  \vert \vert p(t,x) \vert \vert_1 ^2 \big \}. \label{Ineq:Front_weight_est}
\end{aligned}
\end{align}

Since $w(x)$ grows exponentially as $x \rightarrow - \infty$, we take a different approach for $x \leq 0$. This is where we need the a-priori estimate:

\begin{proposition} \label{Prop:A_priori_estimates}
Let $A(t,x),I(t,x)$ be a non-negative solution of the PDE \eqref{Eq:Perturbed_PDE} in the moving frame $x = z -ct$, for a speed $c > 0$. Assume that there exist constants $K, \delta, \mu_0 >0$, such that the initial data fulfill
\begin{align}
I(0, x) & \geq 1 + \delta && \qquad \forall x \leq 0, 
\nonumber \\
A(0,x) & \leq K e^{\mu_0 x} &&\qquad \forall x \leq 0, \\
A(0,x) + I(0,x) & \leq K &&\qquad \forall x \in \mathds{R}. \nonumber
\intertext{Moreover, assume that for some time $t \in (0, \infty]$, it holds that}
I(s, x = 0)  & \geq 1 + \delta > 1 && \qquad \forall s \in [0,t).
\end{align}
Then, there exist $C, \zeta >0$ that are independent of $t$, such that:
\begin{align}
\forall s \in [0,t), \, x \leq 0: \quad & i) & \quad I(s,x) \geq 1 + \delta, \label{Bound:A_pr_I_min} \\ 
& ii) & \quad A(s,x) \leq C e^{\zeta x}. \label{Bound:A_pr_A_max}
\end{align}
\end{proposition}
The proof and the used Feynman-Kac formula are standard, given in Appendix \ref{Sec:A_priori_estimates}. The way we apply this result for controlling the perturbations is new. For the wave it holds that $\lim_{x \rightarrow - \infty}i(x) >1$, where $i(x)$ is monotone. We first shift the wave such that $i(0) = 1 + 2 \delta > 1$. Then, it suffices to control
\begin{align}
\vert \tilde{I}(t,x=0) \vert \leq \delta \qquad \forall t \geq 0, \label{Eq:Prelim_Explanation_x0}
\end{align}
and in view of Proposition \ref{Prop:A_priori_estimates}, there exist constants $C, \zeta >0$ such that
\begin{align}
\vert \tilde{A}(t,x) \vert \leq A(t,x) &\leq C e^{\zeta x} \qquad \forall x \leq 0, \, t \geq 0,
\end{align}
since $A \geq 0$. We re-substitute $w u = \tilde{A}$ to estimate
\begin{align}
\begin{aligned}
& \Big \vert \int_{-\infty}^0 w (x) u(t,x) \cdot \big ( u(t,x) + v(t,x) \big ) \Big \vert \, dx
\\
& \leq \int_{-\infty}^0 \vert \tilde{A}(t,x) \vert \cdot \vert u(t,x) + v(t,x) \vert \, dx
\\
& \leq \sup_{x \leq 0} \Big \{  \vert u(t,x) \vert  +  \vert v(t,x) \vert   \Big \} \cdot \int_{-\infty}^0 \vert \tilde{A}(t,x) \vert \, dx
\\
& \leq C \cdot \sup_{x \leq 0} \big \{  \vert \vert p(t,x) \vert \vert_1   \big \}.
\end{aligned} \label{Ineq:Back_weight_est}
\end{align}

Given both Estimates \eqref{Ineq:Front_weight_est} and \eqref{Ineq:Back_weight_est}, the resulting semi-group estimates are quite standard, though a bit lengthy due to the different cases that need to be dealt with, presented in Section \ref{Sec:Long_time_estimates}.

\subsection{An appropriate function space}
\label{Sec:Prelim_Discussion_Space}

Here and in the following, let $\mathcal{L}$ \eqref{Eq:Linear_Operator} be a densely defined operator
\begin{align}
\mathcal{L}: H^2(\mathds{R}) \times H^2(\mathds{R})  \rightarrow L^2(\mathds{R}) \times L^2(\mathds{R}), \label{Def:Spaces}
\end{align}
where $H^2$ is the $L^2$ Sobolev space. For $\lambda \in \mathds{C}$, consider the Eigenvalue problem $\lambda \cdot (u,v) = \mathcal{L} (u,v)$. With the help of the auxiliary variables
\begin{align}
U := \begin{pmatrix} u \\ u' \\ v \\ v' \end{pmatrix},
\end{align}
the equation $\lambda \cdot (u,v) = \mathcal{L} (u,v)$ can be rewritten as an equivalent linear system of ODEs:
\begin{align}
U' = M(x,\lambda) \cdot U, \label{Def:Operator_as_ODE}
\end{align}
for a matrix $M(x,\lambda)$ given by
\begin{align}
M(x,\lambda) & := \begin{pmatrix}
0 & 1 & 0 & 0 \\
\lambda + \xi_u(x) & -(c + 2\frac{w'(x)}{w(x)}) & a(x) & 0 \\
0 & 0 & 0 & 1  \\
-\frac{2a(x)+i(x)+r}{d} & 0 & \frac{\lambda + \xi_v(x)}{d} & - \frac{c+2d \frac{w'(x)}{w(x)}  }{d}
\end{pmatrix}, \label{Def:Matrix_general_case} \\
\xi_u(x) & := 2a(x)+i(x)-1 -c \frac{w'(x)}{w} - \frac{w''(x)}{w(x)}, \\
\xi_v(x) & := -a(x) - c \frac{w'(x)}{w(x)} - d \frac{w''(x)}{w(x)}. 
\end{align}

Given a pair of exponents $\alpha_\pm  >0$, we fix a smooth weight $w(x) >0$, subject to the conditions
\begin{align}
w(x) := \begin{cases}
e^{- \alpha_+ x} & \quad x \geq 1, \\
e^{- \alpha_- x} & \quad x \leq -1, \\
1  & \quad x = 0. \label{Def:Weight}
\end{cases}
\end{align}
Note that this weight is bounded for $x \geq 0$, and unbounded for $x \leq 0$.

\begin{figure}[h]
	\vspace{0.3cm}
 		\centering
 		\begin{minipage}[c]{0.40\textwidth}
\begin{picture}(100,100)
\put(0,0){\includegraphics[width=\textwidth]{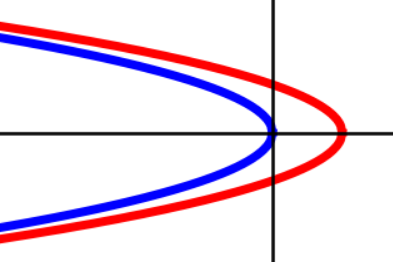}}
	\put(20,53){\scalebox{1.5}{$\mathfrak{Re}$}}
	\put(103,95){\scalebox{1.5}{$\mathfrak{Im}$}}
	\put(120, 65){{\scalebox{2}{${\color{red}\Sigma_\nu}$}}}
	\put(120, 20){{\scalebox{2}{${\color{blue}\Sigma_\psi}$}}}
\end{picture} 	
 		\end{minipage}
 		\hspace{0.5cm}
 		\begin{minipage}[c]{0.40\textwidth}
\begin{picture}(100,100)
 		\put(0,0){\includegraphics[width=\textwidth]{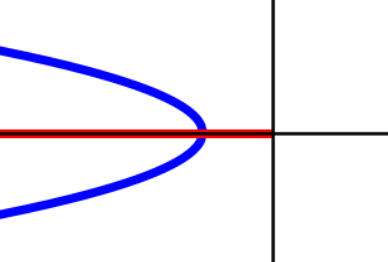}}
	\put(20,53){\scalebox{1.5}{$\mathfrak{Re}$}}
	\put(104,95){\scalebox{1.5}{$\mathfrak{Im}$}}
	\put(70, 65){{\scalebox{2}{${\color{red}\Sigma_\nu}$}}}
	\put(70, 20){{\scalebox{2}{${\color{blue}\Sigma_\psi}$}}}
\end{picture} 	
 		\end{minipage}
 	\vspace{0.4cm}
	\caption{The two critical parts of the essential spectrum of $\mathcal{L}$ as in Proposition \ref{Prop:Ess_Stability}, defined in \eqref{Eq:Ess_spec_line_1} and \eqref{Eq:Ess_spec_parabola2}. On the left, the essential spectrum is unstable in an unweighted space. On the right, the essential spectrum is stabilized via an exponential weight, with Exponents \eqref{Ineq:Stabil_upper_bound}.}\label{Fig:Spectrum_Shift}
\end{figure}
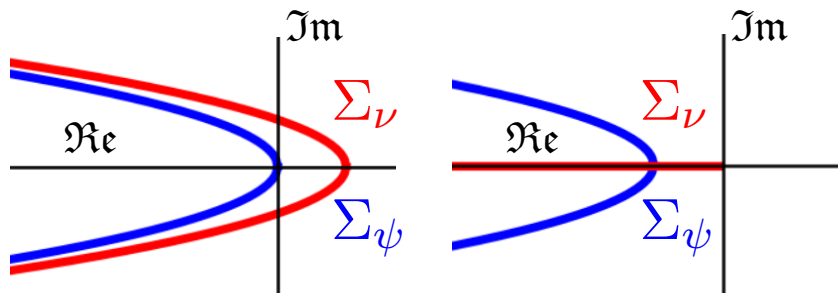

\begin{proposition} \label{Prop:Ess_Stability}
For $d \in (0,1)$ and exponents
\begin{align}
\alpha_- \in ( 0, 1), \quad \alpha_+ = 1, \label{Ineq:Stabil_upper_bound}
\end{align}
consider an exponential weight $w(x)$ as in \eqref{Def:Weight}. Then, for a critical invading front, the essential spectrum of the linear operator $\mathcal{L}$ \eqref{Def:Operator_as_ODE},
\begin{align}
\mathcal{L}: H^2(\mathds{R}) \times H^2(\mathds{R})  \rightarrow L^2(\mathds{R}) \times L^2(\mathds{R}),
\end{align}
is bounded to the right by the union of three parabolas \eqref{Eq:Ess_spec_parabola1}, \eqref{Eq:Ess_spec_parabola2} and \eqref{Eq:Ess_spec_parabola3}, that lie in the strict negative half-plane, except for the half-line
\begin{align}
\Sigma_{\nu} &= \big \{
\lambda \in \mathds{C} \big \vert \, \mathfrak{Re}\lambda \leq 0, \mathfrak{Im} \lambda = 0
\big \}.
\end{align}
\end{proposition}

The proof of Proposition \ref{Prop:Ess_Stability} will be given below. The parabola which we need to stabilize with an appropriate weight and the set $\Sigma_\eta$ are depicted in Figure \ref{Fig:Spectrum_Shift}. Precisely this type of essential spectrum is considered in \cite{FAYE_lotke_volterra_critical_stability}, and we can use their estimates regarding the long-time behavior of $\mathcal{L}$. So far, we could not analyze the discrete spectrum of $\mathcal{L}$ analytically, and need to make the following critical\\

\textbf{Assumption on the discrete spectrum:} In the Setting of the above Proposition \ref{Prop:Ess_Stability}, let $\Sigma_{d}$ be the discrete spectrum of $\mathcal{L}$. Assume that there exist $\delta_0, \delta_1 >0$, such that for the region
\begin{align}
& \Omega := \Big \{ 
\lambda \in \mathds{C} \big \vert \, \mathfrak{Re} \lambda \geq - \delta_0 - \delta_1 \cdot  \vert \mathfrak{Im} \lambda  \vert
\Big \} \label{Eq:Omega_Set}
\intertext{it holds that}
& \Omega \cap \Sigma_{d} = \emptyset.
\label{Ass:Point_spectrum}
\end{align}
We verify this assumption numerically, presented in Appendix \ref{Sec:Numerical_spectrum}. Note that we require $0 \notin \Sigma_d$. For an invading front $a(x), i(x)$, it is easy to see that the symbolic equation $\mathcal{L}(a',i') = 0$ holds. However, the critical front behaves like $x e^{-x}$ as $x \rightarrow + \infty$, so its derivative is not an element of the considered weighted space.

\begin{proof}[Proof of Proposition \ref{Prop:Ess_Stability}]

For any non-negative traveling wave, $a(x)$ vanishes as $x \rightarrow \pm \infty$, and $i(x)$ converges. Moreover, both converge exponentially fast at both ends. The same applies to the matrices $M(x,\lambda)$, which we call exponentially localized. In this setting it is a well-known result \cite{SANDSTEDE_stability_traveling, Gardner_Zumbrun_Gap_Lemma} that the essential spectrum of $\mathcal{L}: H^2 \rightarrow L^2$ is bounded to the right by the spectrum of the limit matrices $M(\pm \infty, \lambda)$: those values of $\lambda \in \mathds{C}$, for which one of $M(\pm \infty, \lambda)$ is defect. Denoting as $\mathcal{E}(M)$ the Eigenvalues of a matrix $M$, the essential spectrum of $\mathcal{L}$ is thus bounded by the set
\begin{align}
\begin{aligned}
M_{Ess} =
& \, \Big \{ \lambda \in \mathds{C} \Big \vert \, \exists \, \mu \in \mathcal{E}\big( M(+\infty, \lambda) \big): \, \mathfrak{Re}(\mu) = 0 \Big \}
 \\ \cup
& \, \Big \{ \lambda \in \mathds{C} \Big \vert \, \exists \, \mu \in \mathcal{E}\big( M(-\infty, \lambda) \big): \, \mathfrak{Re}(\mu) = 0 \Big \}.
\end{aligned}
\end{align}
The region to the right of this set is referred to as \textit{region of consistent splitting}, since the dimension of the unstable spectrum of $M(-\infty, \lambda)$ and the dimension of the stable spectrum of $M(+\infty, \lambda)$ add up to the full dimension of the system. In the following, we focus on the critical invading fronts, but in principle, technique and result can be applied to the other critical traveling waves of the system (with minimal speed depending on the chosen limits). 

For $d\in (0,1)$, consider now an invading front $a(x), i(x)$ with speed $c = 2$. We first analyze the matrix $M(+ \infty, \lambda)$. Given that $\ipinf =0$ and with Weight \eqref{Def:Weight}, the limit of $M(x,\lambda)$ \eqref{Def:Matrix_general_case} at $x = + \infty$ is given by
\begin{align}
\begin{pmatrix}
0 & 1 & 0 & 0 \\
\lambda  -1 + 2 \alpha_+ - \alpha_+^2  & -2 +2 \alpha_+ & 0 & 0 \\
0 & 0 & 0 & 1  \\
-\frac{r}{d} & 0 & \frac{\lambda + 2 \alpha_+ - d\alpha_+^2 }{d} & - \frac{2 - 2d \alpha_+  }{d}
\end{pmatrix}.
\end{align}
This matrix has two pairs of Eigenvalues:
\begin{align}
\nu_\pm(\lambda, \alpha_+) & =  -1 \pm \sqrt{\ipinf + \lambda} + \alpha_+, \\
\eta_\pm(\lambda, \alpha_+) &= \frac{1}{d} \big ( - 1 \pm \sqrt{d\lambda +1} \big ) + \alpha_+.
\end{align}
We define the set
\begin{align}
\Sigma_{\nu} &:= \big \{ 
\lambda \in \mathds{C} \big  \vert  \, \mathfrak{Re} \big ( \nu_+(\lambda) \big ) = 0
\big \}. \label{Eq:Ess_spec_line_1}
\end{align}
For $w \equiv 1$ (if $\alpha_+ = 0$), the set $\Sigma_{\nu}$ is a parabola that goes into the right half-plane, as depicted in Figure \ref{Fig:Spectrum_Shift}. In order to stabilize it, we need to operate in a space with weight
\begin{align}
w(x) = e^{ - \frac{c}{2} x} = e^{-x}, \qquad \forall x \geq 1.
\end{align}
This is the only possible choice for stabilizing the front of a critical FKPP-wave, a classical result, check \cite{SATTINGER_stability, Faye2018_critical_FKPP}. For $\alpha_+ = 1$, the Eigenvalues of $M(+ \infty, \lambda)$ are then given by
\begin{align}
\nu_\pm(\lambda,1) = \pm \sqrt{\lambda}, \quad \eta_\pm(\lambda,1) = \frac{1}{d} \Big ( 
d-1 \pm \sqrt{ d \lambda +1 }
\Big ).
\end{align}
The pair $\nu_\pm$ creates an essential spectrum that lies on the negative real axis and goes up to the origin:
\begin{align}
\Sigma_{\nu} &= \Big \{
\lambda \in \mathds{C} \big \vert \, \mathfrak{Re}\lambda \leq 0, \mathfrak{Im} \lambda = 0
\Big \}. \label{Eq:Ess_spec_line}
\intertext{Since $d \in (0,1)$, the Eigenvalue $\eta_-(\lambda)$ does not touch the imaginary axis, only the Eigenvalue $\eta_+(\lambda)$ crosses the imaginary axis along a parabola that lies in the strict negative half-plane:}
\Sigma_{\eta} &= \Big \{ \lambda \in \mathds{C} \big \vert \, \mathfrak{Re}\lambda = \frac{1}{d} \big ( 2 (1-d)^2 -1 - \vert d\lambda+1 \vert^2 \big )
\Big \}.
\label{Eq:Ess_spec_parabola1}
\end{align}

We now consider the other limit, $M(- \infty, \lambda)$, given by
\begin{align}
\begin{pmatrix}
0 & 1 & 0 & 0 \\
\lambda + \iminf -1 + 2 \alpha_- - \alpha_-^2  & -2 +2 \alpha_- & 0 & 0 \\
0 & 0 & 0 & 1  \\
-\frac{r+\iminf}{d} & 0 & \frac{\lambda + 2 \alpha_- - d\alpha_-^2 }{d} & - \frac{2 - 2d \alpha_-  }{d}
\end{pmatrix},
\end{align}
for some $\iminf \in (1,2)$, and Eigenvalues
\begin{align}
\sigma_\pm(\lambda, \alpha_-) & =  -1 \pm \sqrt{\iminf + \lambda} + \alpha_-, \\
\psi_\pm(\lambda, \alpha_-) & = \frac{1}{d} \big ( - 1 \pm \sqrt{d\lambda +1} \big ) + \alpha_-.
\end{align}
In order to stabilize $\psi_+$ away from the origin, we can pick any $\alpha_- >0$, leading to the parabola
\begin{align}
\Sigma_{\psi} &:= \Big \{
\lambda \in \mathds{C} \big \vert \, \mathfrak{Re}\lambda = \frac{1}{d} \big ( 2(1-d\alpha_-)^2 - 1 - \vert \lambda+1 \vert \big )
\Big \}. \label{Eq:Ess_spec_parabola2}
\intertext{As long as $\alpha_- < 2 / d$, the Eigenvalue $\psi_-$ does not touch the imaginary axis. The Eigenvalue $\sigma_+$ crosses the imaginary axis along a parabola that lies in the strict negative half-plane:}
\Sigma_{\sigma} &:= \Big \{
\lambda \in \mathds{C} \big \vert \, \mathfrak{Re}\lambda = - \iminf - 2(1-\alpha_-)^2 - \vert \iminf + \lambda \vert
\Big \}. \label{Eq:Ess_spec_parabola3}
\end{align}
Lastly, we check that we do not overstabilize $\sigma_-$. It holds that
\begin{align}
\mathfrak{Re} \big ( \sigma_-(\lambda) \big ) = 0 \, \Leftrightarrow \, \alpha_- -1 =  \mathfrak{Re} \big ( +\sqrt{ \iminf + \lambda } \big ).
\end{align}
The set $\{ \mathfrak{Re} \big ( \sigma_-(\lambda) \big ) = 0 \}$ is obviously empty as long as
\begin{align}
\alpha_- < 1.
\end{align}
\end{proof}

\subsection{Estimating the long-time behavior}
\label{Sec:Long_time_estimates}

We prove the stability of the critical invading front using estimates by Fay\'{e} and Holzer \cite{FAYE_lotke_volterra_critical_stability}. By controlling the behavior of the Evans function (and thereby of the point-wise Green's function) in the neighborhood of the essential spectrum, they prove the following decay of the temporal Green's function:

\begin{theorem}[cf. Prop. 4.1 and Lemma 5.1 in \cite{FAYE_lotke_volterra_critical_stability}]  \label{Thm:Faye_Green_bound}
For $\mathcal{L}$ with essential spectrum as in Proposition \ref{Prop:Ess_Stability}, and given that the discrete spectrum of $\mathcal{L}$ fulfills Assumption \eqref{Ass:Point_spectrum}, the temporal Green's function $\mathcal{G}(t,x,y)$ for $\partial_t (u,v) = \mathcal{L}(u,v)$ satisfies the following estimates: there exists constants $C, \kappa > 0$, such that for all pairs of indices $i,j \in \{1,2\}$ and for all $x,y \in \mathds{R}$:
\begin{align}
\big \vert \mathcal{G}_{ij}(t,x,y) \big \vert & \leq C \frac{1}{t^{1/2}} e^{- \frac{ \vert x-y \vert ^2}{\kappa t} } &&\quad \forall t < 1, \label{Ineq:Green_pointwise} \\
\int_{\mathds{R}} \big \vert\mathcal{G}_{ij}(t,x,y) \big \vert \cdot \vert h(y) \vert \, dy  & \leq C \cdot \vert \vert h \vert \vert_\infty &&\quad \forall t < 1. \label{Ineq:Green_infty}
\end{align}
Moreover, for all $t\geq 1$ and $x \in \mathds{R}$:
\begin{align}
\int_{\mathds{R}}  \big \vert \mathcal{G}_{i,j}(t,x,y)  \big \vert \cdot \vert h(y) \vert \, dy  \leq C \cdot \frac{1 +  \vert x \vert }{(1+t)^{3/2}} \int_\mathds{R} (1+ \vert y \vert ) \cdot  \vert h(y) \vert  \, dy. \label{Ineq:Green_L1}
\end{align}
\end{theorem}

In \cite{FAYE_lotke_volterra_critical_stability}, the authors prove a local stability result for the critical traveling waves of a Lotka-Volterra model with two species. They consider an integrable weight $w(x)$ and follow the reasoning behind Inequality \eqref{Ineq:Front_weight_est}. We extend their proof by using the a-priori bounds that we have for the left tail of the system, see \eqref{Ineq:Back_weight_est}. The following Lemma gives control of the resulting integrals over time:

\begin{lemma}[Lemma 2.3 in \cite{Xin_Root_Integral_Lemma}] \label{Lem:Root_Integral_Lemma}
Let $\alpha, \beta, \gamma,t >0$ with $\alpha \leq \beta + \gamma -1$. If either $\alpha \leq \beta,  \gamma \neq 1$, or $\alpha < \beta, \gamma = 1$, then there exists a constant $C$ such that
\begin{align}
\int_0^{\frac{t}{2}} \frac{1}{(1+t-s)^\beta} \frac{1}{(1+s)^\gamma} \, ds &\leq C \frac{1}{(1+t)^\alpha}.
\intertext{Similarly, if either $\alpha \leq \gamma, \beta \neq 1$, or $\alpha < \gamma, \beta = 1$, then}
\int_{\frac{t}{2}}^t \frac{1}{(1+t-s)^\beta} \frac{1}{(1+s)^\gamma} \, ds &\leq C \frac{1}{(1+t)^\alpha}.
\end{align}
\end{lemma}

We can now prove the stability of the critical front:

\begin{proof}[\textbf{Proof of Theorem \ref{Thm:large_time_decay}}]
We adopt the notation introduced in Section \ref{Sec:Notation_Pert}. That is, we consider the vector $p = (u,v)$, and write $ \vert p(t,x) \vert  =  \vert \vert p(t,x) \vert \vert_1 $. Note that since $ \vert wu \vert  =  \vert \tilde{A} \vert  \leq 1$, all reaction-terms of the perturbation System \eqref{Eq:Perturbed_Oparator_corrected} are at most linear in $ \vert p \vert $. Thus, by a standard Gronwall estimate and a fixed-point argument, a unique smooth solution (that can grow at most exponentially over time) of System \eqref{Eq:Perturbed_Oparator_corrected} exists for arbitrarily long times, check e.g. chapter 14 in \cite{Smoller_Shocks_Reaction_Diffusion}, and Thm. 46.4 and concluding remarks in \cite{Sell_Dynamics_Evolutionary_Equations}.

We will prove that
\begin{align}
\Theta(t) := \sup_{s \leq t} \sup_{x \in \mathds{R}} \frac{(1+s)^{3/2}}{1+ \vert x \vert }  \vert p(s,x) \vert \label{Eq:Theta_def}
\end{align}
is bounded uniformly in $t \geq 0$, if the initial data is small enough. We introduce a border $B \leq 0$ to be specified later, and split the Duhamel Formula \eqref{Eq:Duhamels_formula}:
\begin{align}
 \vert p(t,x) \vert  & \leq \Big \vert \int_{\mathds{R}}  \mathcal{G}(t,x,y) p(0,y) \, dy
\label{Eq:Duhamel_const_part} \Big \vert
\\
& + \Big \vert 
\int_0^t ds \int_{-\infty}^B \mathcal{G}(t-s,x,y) N(p)(s,y) \, dy \Big \vert
\label{Eq:Duhamel_neg_part}
\\
& + \Big \vert \int_0^t ds \int_{B}^{+\infty} \mathcal{G}(t-s,x,y) N\big (p)(s,y) \, dy \Big \vert.
\label{Eq:Duhamel_pos_part}
\end{align}
For $y \geq B$, we use the exponential decay of $w(y)$, whereas for $y \leq B$, we use the exponential decay of $A(t,y)$, see \eqref{Ineq:Back_weight_est}. Theorem \ref{Thm:Faye_Green_bound} yields different results for $\mathcal{G}$ for $t \leq 1$ and $t \geq 1$, thus we also differentiate the above terms for $t-s \leq 1$ and $t-s \geq 1$. In the following, we let $C$ be a universal constant that does not depend on $B$, whereas $C_B$ will be a universal constant that does.

We shift the given traveling wave $a(x), i(x)$, such that there exists a $\delta_I >0$ with
\begin{align}
i(x) \geq 1 + 2 \delta_I \qquad \forall x \leq 0. \label{Eq:Main_Thm_I_Nec}
\end{align}
We assume that $ \vert  \vert \tilde{I}(0,x) \vert  \vert _\infty \leq \delta_I$, then
\begin{align}
I(0,x) = i(x) + \tilde{I}(0,x) \geq 1 + \delta_I \qquad \forall x \leq 0. \label{Eq:Main_Thm_I_Nec1}
\end{align}
For the moment, we also assume that
\begin{align}
\vert \tilde{I}(t,x=0)\vert = \vert v(t,x=0) \vert \leq \delta_I  \qquad  \forall t \geq 0, \label{Eq:Main_Thm_I_Nec2}
\intertext{which in particular implies that}
I(t,x=0) \geq 1 + \delta_I  \qquad \forall t \geq 0.
\end{align}
We will verify \eqref{Eq:Main_Thm_I_Nec2} a posteriori, by proving that the perturbations stay small. Given \eqref{Eq:Main_Thm_I_Nec1} and \eqref{Eq:Main_Thm_I_Nec2}, we can apply Proposition \ref{Prop:A_priori_estimates}, the Feynman-Kac formula. As a result, there exists an exponent $\zeta >0$, such that
\begin{align}
 \vert N(p)(t,x) \vert  &= \big \vert w(x) u(t,x) \cdot \big ( u(t,x) + v(t,x)  \big ) \big \vert \nonumber
\\
&  \leq A(t,x) \cdot \big \vert p(t,x) \big \vert
\\
&  \leq C e^{ \zeta x} \cdot \big \vert p(t,x) \big \vert && \forall x \leq B \leq 0. \label{Bound:Main_proof_back_bound} \nonumber
\intertext{In contrast, for $x \geq B$, we will use the general bound}
 \vert N(p)(s,x) \vert  & \leq w(x) \cdot  \vert p(t,x) \vert ^2 && \forall x \in \mathds{R}.
\end{align}

Now choose any $\epsilon \leq \delta_I$. We estimate \eqref{Eq:Duhamel_const_part}, \eqref{Eq:Duhamel_neg_part} and \eqref{Eq:Duhamel_pos_part} separately, starting with\\

\textbf{1) The long-time case $\boldsymbol{t \geq 1}$}\\
By Estimate \eqref{Ineq:Green_L1} of Theorem \ref{Thm:Faye_Green_bound}, the linear Part \eqref{Eq:Duhamel_const_part} is bounded by
\begin{align}
\Big \vert \int_{\mathds{R}} G(t,x,y) p(0,y) dy \Big \vert \leq
C \frac{1+\vert x \vert }{(1+t)^{3/2}} \int_\mathds{R} (1+\vert y\vert) \cdot \vert p(0,y) \vert dy. \label{Eq:Duhamel_const_part_done}
\end{align}
For the moment, we only require that
\begin{align}
P(0) := \vert \vert p(0,x) \cdot (1+ \vert x \vert ) \vert \vert_{L^1(\mathds{R})}  \leq 1, \label{Eq:P0_Def}
\end{align}
then the above expression is of order $(1+\vert x\vert)/(1+t)^{3/2}$.

Regarding the nonlinear part, we first consider the case $t-s \leq 1$. For the back of the wave, we use the exponential Decay \eqref{Bound:Main_proof_back_bound} and Estimate \eqref{Ineq:Green_infty} of Theorem \ref{Thm:Faye_Green_bound}:
\begin{align}
& \int_{t-1}^t ds \int_{-\infty}^{B} \big \vert \mathcal{G}(t-s,x,y) \big \vert \cdot \vert N (p)(s,y) \vert dy \nonumber
\\
& \leq C \int_{t-1}^t ds \int_{-\infty}^{B} \big \vert \mathcal{G}(t-s,x,y) \big \vert \cdot e^{ \zeta y} \vert p(t,y) \vert \, dy \nonumber
\\
& \leq C \int_{t-1}^t ds \int_{-\infty}^{B} \big \vert \mathcal{G}(t-s,x,y)\big \vert \cdot e^{ \zeta y} \Theta(s) \frac{1+\vert y \vert }{(1+s)^{3/2}}  \, dy
\\
& \leq C \Theta(t) \frac{1}{(1+t)^{3/2}} \int_{-\infty}^{B} \big \vert \mathcal{G}(t-s,x,y)\big \vert \cdot e^{ \zeta y} (1+\vert y \vert ) \, dy \nonumber \\
& \leq C \Theta(t) \frac{1}{(1+t)^{3/2}} \cdot \sup_{y \leq B} \big \{ e^{\zeta y} (1+\vert y \vert ) \big \}. \nonumber
\label{Eq:Duhamel_neg_part_done_Small_st}
\end{align}
Note that by choosing $B \leq 0$ sufficiently small, the supremum in the last term can be made arbitrarily small.

For $x \geq B$, we use the fact $w(x) \cdot (1+\vert x \vert)^2$ is bounded in $ \vert  \vert . \vert  \vert _\infty$. We apply Estimate \eqref{Ineq:Green_infty} of Theorem \ref{Thm:Faye_Green_bound}:
\begin{align}
& \int_{t-1}^t ds \int_{B}^\infty \big \vert \mathcal{G}(t-s,x,y)\big \vert \cdot \vert N (p)(s,y) \vert \, dy \nonumber \\
&  \leq \int_{t-1}^t ds \int_{B}^\infty \big \vert \mathcal{G}(t-s,x,y) \big \vert \cdot \Theta(s)^2 \frac{(1+ \vert y \vert )^2}{(1+s)^{3}} w(y) \, dy \nonumber \\
 & \leq C \Theta(t)^2 \frac{1}{(1+t)^3} \int_{t-1}^t ds \int_{B}^\infty \big \vert \mathcal{G}(t-s,x,y)\big \vert \cdot (1+ \vert y \vert )^2 w(y) \, dy \\
& \leq C_B \Theta(t)^2 \frac{1}{(1+t)^3}. \nonumber
\label{Eq:Duhamel_pos_part_done_Small_st}
\end{align}

Now consider $t-s \geq 1$. Regarding $y \leq B$, we again use Bound \eqref{Bound:Main_proof_back_bound}, and Estimate \eqref{Ineq:Green_L1} of Theorem \ref{Thm:Faye_Green_bound}:
\begin{align}
&  \int_0^{t-1} ds \int_{-\infty}^{B} \big \vert \mathcal{G}(t-s,x,y)\big \vert \cdot  \vert  N (p)(s,y)  \vert  \,  dy  \nonumber \\
& \leq C
\int_0^{t-1} ds \int_{-\infty}^B \big \vert \mathcal{G}(t-s,x,y) \big \vert \cdot e^{ \zeta y}  \vert p(t,y) \vert  \, dy \nonumber \\
& \leq C \int_0^{t-1} ds  \int_{-\infty}^B \big \vert \mathcal{G}(t-s,x,y) \big \vert \cdot e^{ \zeta y} \Theta(s) \frac{1+ \vert y \vert }{(1+s)^{3/2}}  \, dy  \\
& \leq C \Theta(t) \cdot (1+ \vert x \vert ) \int_0^t \frac{1}{(1+t-s)^{3/2}} \frac{1}{(1+s)^{3/2}} \, ds \int_{-\infty}^B e^{\zeta y} (1+ \vert y \vert )^2 \, dy \nonumber
\intertext{We apply Lemma \ref{Lem:Root_Integral_Lemma} to bound the integral over time:}
& \leq C \Theta(t) \frac{1+ \vert x \vert }{(1+t)^{3/2}} \int_{-\infty}^B e^{\zeta y} (1+ \vert y \vert )^2 \, dy. \label{Eq:Duhamel_neg_part_done_Big_st}
\end{align}
Again, note that we can make the last integral as small as we want if we shift $B$ appropriately.

Regarding $y \geq B$, by Estimate \eqref{Ineq:Green_L1} of Theorem \ref{Thm:Faye_Green_bound}:
\begin{align}
&  \int_0^{t-1} ds \int_{B}^{+\infty} \big \vert \mathcal{G}(t-s,x,y)\big \vert \cdot \vert N (p)(s,y) \vert \,  dy  \nonumber \\
& \leq 
\int_0^{t-1} ds \int_B^\infty \big \vert \mathcal{G}(t-s,x,y)\big \vert \cdot \Theta(s)^2 \frac{(1+ \vert y \vert )^2}{(1+s)^{3}} w(y) \, dy \nonumber \\
& \leq C \Theta(t)^2  (1+ \vert x \vert ) \int_0^t \frac{1}{(1+t-s)^{3/2}} \frac{1}{(1+s)^3} \, ds  \int_B^\infty (1+ \vert y \vert )^3 w(y) \, dy \nonumber \\
& \leq C_B  \Theta(t)^2 (1+ \vert x \vert ) \int_0^t \frac{1}{(1+t-s)^{3/2}} \frac{1}{(1+s)^3} \, ds
\intertext{We apply Lemma \ref{Lem:Root_Integral_Lemma} to bound the integral over time, yielding}
& \leq C_B \Theta(t)^2 \frac{1+ \vert x \vert }{(1+t)^{3/2}}.  \label{Eq:Duhamel_pos_part_done_Big_st}
\end{align}

Now, for all $t \geq 1$, combining our estimates \eqref{Eq:Duhamel_const_part_done}, \eqref{Eq:Duhamel_neg_part_done_Small_st}, \eqref{Eq:Duhamel_pos_part_done_Small_st}, \eqref{Eq:Duhamel_neg_part_done_Big_st},  \eqref{Eq:Duhamel_pos_part_done_Big_st} results in
\begin{align}
 \vert p(t,x) \vert  & \leq C P(0) \frac{1+ \vert x \vert }{(1+t)^{3/2}}
\\
& + C \Theta(t) \frac{1}{(1+t)^{3/2}} \cdot \sup_{y \leq B} \big \{ e^{\zeta y} (1+ \vert y \vert ) \big \} \label{Eq:Thate_linear1}
\\
& + C_B \Theta(t)^2 \frac{1}{(1+t)^{3/2}}
\\
& + C \Theta(t) \frac{1+ \vert x \vert }{(1+t)^{3/2}} \int_{-\infty}^B e^{\zeta y} (1+ \vert y \vert )^2 \, dy \label{Eq:Thate_linear2}
\\
& + C_B \Theta(t)^2 \frac{1+ \vert x \vert }{(1+t)^{3/2}}. 
\end{align}
Next, choose $B$ sufficiently small such that the constants in \eqref{Eq:Thate_linear1} and \eqref{Eq:Thate_linear2}, the terms that are linear in $\Theta$, are both bounded by $1/8$. Then, we can simplify the above to
\begin{align}
 \vert p(t,x) \vert  & \leq C P(0) \frac{1+ \vert x \vert }{(1+t)^{3/2}} + \frac{1}{4} \Theta(t) \frac{1+ \vert x \vert }{(1+t)^{3/2}} + C_B \Theta(t)^2 \frac{1+ \vert x \vert }{(1+t)^{3/2}}. 
\intertext{Inserting our definition of $\Theta$, see \eqref{Eq:Theta_def}, we get}
\Theta(t) & \leq C P(0) + \frac{1}{4} \Theta(t) + C_B \Theta(t)^2 \qquad \forall t \geq 1. \label{Eq:T_Large_complete_bound}
\end{align}

\textbf{2) The short-time case $\boldsymbol{t < 1}$}\\
Estimate \eqref{Ineq:Green_infty} of Theorem \ref{Thm:Faye_Green_bound} yields the following bound for the linear Part \eqref{Eq:Duhamel_const_part}:
\begin{align}
\int_{\mathds{R}} \big \vert \mathcal{G}(t,x,y)\big \vert \cdot  \vert  p(0,y) \vert  \, dy   \leq C \cdot \, \vert   \vert   p(0,x) \vert  \vert_\infty. \label{Eq:Duh_short_1}
\end{align}
For the nonlinear part, we again split the expression into two parts. This time, we use the point-wise Estimate \eqref{Ineq:Green_pointwise} of Theorem \ref{Thm:Faye_Green_bound}, valid for short times:

\begin{align}
& \int_0^{t-1} ds \int_{-\infty}^{B} \big \vert \mathcal{G}(t-s,x,y)\big \vert \cdot \vert N (p)(s,y) \vert \, dy  \nonumber
\\
& \leq  \int_0^{t-1} ds \int_{-\infty}^{B} \big \vert \mathcal{G}(t-s,x,y) \big \vert \cdot e^{\zeta y} \vert p(t,y) \vert \, dy \nonumber
\\
& \leq \Theta(t) \frac{1}{(1+t)^{3/2}} \cdot   \int_0^{t-1} ds \int_{-\infty}^{B} \big \vert \mathcal{G}(t-s,x,y) \big \vert \cdot e^{\zeta y} (1+ \vert y \vert ) \, dy  
\\
& \leq C \Theta(t) \int_0^{t-1} ds \int_{-\infty}^{B} \frac{1}{t^{1/2}}  e^{- \frac{ \vert x-y \vert ^2}{\kappa t} }  e^{\zeta y} (1+ \vert y \vert ) \, dy. \nonumber
\intertext{
The last integral is a short-time heat kernel applied to the exponentially decaying function $e^{\zeta x} (1+ \vert x \vert )$. We choose $B$ appropriately such that the entire above expression is bounded by}
& \leq \frac{1}{16} \Theta(t). \label{Eq:Duh_short_2}
\end{align}

Considering $y \geq B$:

\begin{align}
& \int_0^{t-1} ds \int_{B}^{+\infty} \big \vert \mathcal{G}(t-s,x,y)\big \vert \cdot \vert N (p)(s,y)\vert \,  dy 
	\nonumber \\
& \leq   \int_0^{t-1} ds \int_{B}^{+\infty} \big \vert \mathcal{G}(t-s,x,y) \big \vert \cdot \Theta(s)^2 \frac{(1+ \vert y \vert )^2}{(1+s)^3} \, dy
\nonumber \\
& \leq C \Theta(t)^2 \int_0^{t-1} ds \int_{B}^{+\infty}  \frac{1}{t^{1/2}} e^{- \frac{ \vert x-y \vert ^2}{\kappa t} } (1+ \vert y \vert )^2 \, dy \label{Eq:Duh_short_3} \\
& \leq C_B \Theta(t)^2. \nonumber
\end{align}
In view of \eqref{Eq:Duh_short_1}, \eqref{Eq:Duh_short_2}, \eqref{Eq:Duh_short_3}, we see that for all $t < 1$:
\begin{align}
 \vert p(t,x) \vert  &\leq C \, \vert   \vert   p(0,x) \vert  \vert_{L^\infty(\mathds{R})} + \frac{1}{16} \Theta(t) + C_B \Theta(t)^2.
\intertext{This implies that for small times, by the Definition of $\Theta$ \eqref{Eq:Theta_def}:}
\Theta(t) &\leq C \, \vert   \vert   p(0,x) \vert  \vert_{L^\infty(\mathds{R})} + \frac{1}{4} \Theta(t) + C_B \Theta(t)^2 \qquad \forall t <1. \label{Eq:T_Small_complete_bound}
\end{align}

\textbf{3) Convergence given small initial data}\\
To control both the short-time Bound \eqref{Eq:T_Small_complete_bound} and the large-time Bound \eqref{Eq:T_Large_complete_bound}, we will consider initial data where
\begin{align}
Q(0) := \vert \vert p(0,x) \vert \vert_{L^\infty(\mathds{R})} + \vert \vert p(0,x) \cdot (1+ \vert x \vert ) \vert \vert_{L^1(\mathds{R})}
\end{align}
is sufficiently small. Regarding both the long-time and the short-time case, we first choose a border $B \in \mathds{R}$ such that both \eqref{Eq:T_Small_complete_bound} and \eqref{Eq:T_Large_complete_bound} are valid, resulting in
\begin{align}
\Theta(t) & \leq C Q(0) + \frac{1}{4} \Theta(t) + C_B \Theta(t)^2 \qquad \forall t \geq 0. \label{Eq:Theta_final}
\end{align}
Without loss of generality, we may assume that $C \geq 1$. Consider initial data that are small enough to fulfill
\begin{align}
2 C Q(0) < \epsilon \quad \text{and} \quad 4 C C_B Q(0) < \frac{1}{2},
\end{align}
where $\epsilon \leq \delta_I$, with $\delta_I$ chosen in \eqref{Eq:Main_Thm_I_Nec}. Then, at $t = 0$:
\begin{align}
\Theta(0) = \sup_{x \in \mathds{R}} \frac{ \vert p(0,x) \vert }{1+ \vert x \vert } \leq Q(0) < 2C Q(0) < \epsilon.
\end{align}
For such initial data, our critical Assumption \eqref{Eq:Main_Thm_I_Nec2} holds for small times $t>0$, since $\Theta(t)$ is continuous in $t$. Now suppose that there exists a first time $T\in (0, \infty)$ such that $\Theta(T) = 2CQ(0)$ for the first time. But then, by \eqref{Eq:Theta_final}:
\begin{align}
\Theta(t) &\leq C Q(0) + \frac{1}{4} \Theta(t) + C_B \Theta(t)^2 \nonumber
\\
& \leq C Q(0) + \frac{1}{4} 2 C Q(0) + C_B 4 C^2 Q(0)^2 \nonumber
\\
& \leq C Q(0) + \frac{1}{2} C Q(0) + C Q(0) \big [ 4 C C_B Q(0)  \big]
\\
& < 2C Q(0), \nonumber
\end{align}
a contradiction to our assumption. As a result, it holds that
\begin{align}
\Theta(t) < 2CQ(0) < \epsilon \qquad \forall t \geq 0,
\end{align}
which not only proves that the perturbation decays, but also shows that the necessary Bound \eqref{Eq:Main_Thm_I_Nec2} holds for all $t \geq 0$.
\end{proof}


\section{Construction of the traveling waves}
\label{Sec:Existence_big_section}
\subsection{Overview and notation}

In the following, we refer to the Wave System \eqref{Eq:perturbed_wave} as $S_0$ if $d=0$, and to $S_d$ for $d >0$. We cite the result for the traveling waves of $S_0$, which was the subject of a previous study:

\begin{theorem} [Thm. 1.1 in \cite{Kreten2022}] \label{Old_main_Theorem}
For $d = 0$ and $r \geq 0, c >0$, consider System $S_0$ \eqref{Eq:perturbed_wave}. Set $\oc : = \max \{ 0, 1-c^2/4 \} $. For each pair $\iminf, \ipinf \in \mathds{R}^+_0$ such that
\begin{align}
\ipinf \in [\oc,1), \qquad \iminf = 2 - \ipinf, \label{main_thm_relation}
\end{align}
there exists a unique non-negative traveling wave $a,i \in C^\infty(\mathds{R}, \mathds{R}^2)$ such that
\begin{align}
\lim_{x \rightarrow \pm \infty} & a(x) = 0, \qquad  \lim_{x \rightarrow \pm \infty} i(x) = i_{\pm \infty}. \label{intro_main_them_limits}
\end{align}
The function $i(x)$ is decreasing, whereas $a(x)$ has a unique local and global maximum. If $\frac{c^2}{4} + \ipinf -1  = 0$, then the distance of the front to its limit behaves like $ x \cdot e^{ - \frac{c}{2} x}$ asymptotically as $x \rightarrow + \infty$. If $ \frac{c^2}{4} + \ipinf -1  > 0$, then convergence as $x \rightarrow + \infty$ is purely exponential. Convergence as $x \rightarrow - \infty$ is purely exponential in all cases. The rates of convergence are
\begin{align}
\begin{aligned}
\mu_{- \infty} &= - \frac{c}{2} + \sqrt{ \frac{c^2}{4} + i_{\pm \infty} -1} >0 , \\
\mu_{+ \infty} &=  \frac{c}{2} - \sqrt{ \frac{c^2}{4} + i_{\pm \infty} -1} >0. \label{Eq:Rates_Conv_Thm_old}
\end{aligned}
\end{align}
Moreover, these are all bounded, non-negative, non-constant and twice differentiable solutions of Eq. \eqref{Eq:perturbed_wave} for $d = 0$.
\end{theorem}

\begin{figure}[h]
\vspace{1.2cm}
 		\centering
 		\begin{minipage}[c]{0.45\textwidth}
  \begin{picture}(100,100)
	\put(0,0){\includegraphics[width=0.8\textwidth]{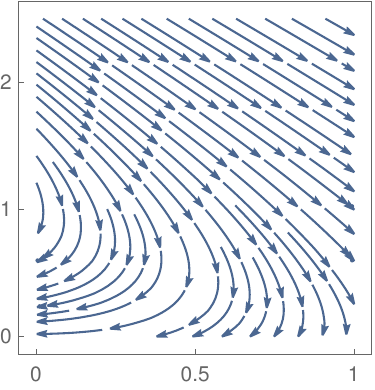}}
	\put(-5,75){\scalebox{1.8}{$i$}}
	\put(90,-5){\scalebox{1.8}{$a$}}
\end{picture}
\end{minipage}
\vspace{0.3cm}
 		\caption{Two-dimensional representation of the family of traveling waves of $S_0$ for $c=2$ and $r=0$. A unique trajectory emerges from each point where $\iminf >1$, and converges to a limit where $\ipinf <1$, where $\iminf + \ipinf = 2$. For $\iminf >2$, the trajectory eventually becomes negative (even though it still seems to converge).}
 		\label{fig:old_wave_phase_plot}
\end{figure}

The bound $\ipinf \geq 1-c^2/4$ is classical for FKPP-fronts. If it is not fulfilled, the solutions spiral around their limit as $x \rightarrow + \infty$, see Section \ref{Sec:Degenerate_Linearization}. Thus, they become negative and are physically irrelevant.

For the entire Section \ref{Sec:Existence_big_section}, keep the phase-plot of $S_0$ in Figure \ref{fig:old_wave_phase_plot} in mind. Qualitatively, we prove that this portrait remains valid for $S_d$: there exists a family of solutions, that continuously vary along the limits $i_{ \pm \infty}$. The only change is a perturbation estimate of type
\begin{align}
\iminf + \ipinf = 2 + \mathcal{O}(d),
\end{align}
which replaces the precise Statement \eqref{main_thm_relation} for $d=0$. Since we focus on the invading fronts, we consider only the case $c \geq 2$, the result is given in Theorem \ref{Prop:Existence_of_a_tr_sol}.

We transform the wave Eq. \eqref{Eq:perturbed_wave} into an equivalent System of ODEs. We denote differentiation w.r.t. $x$ by a prime, and introduce two auxiliary variables $b=a'$ and $j=i'$. For $d \neq 0$, this system in coordinates $(a,b,i,j) \in \mathds{R}^4$ reads
\begin{align}
\frac{d}{dx} \begin{pmatrix}
a \\ b \\ i \\ j
\end{pmatrix}
= \begin{pmatrix}
b \\
a(a+i) -a -cb \\
j \\
- \frac{1}{d} [cj + ra + a(a+i)] \label{Eq:perturbed_wave_ODE}
\end{pmatrix}.
\end{align}
By abuse of notation, we will denote $a'=b$ and $i'=j$, since introducing two auxiliary variables only obfuscates the system.

Section \ref{Sec:Existence_big_section} is organized as follows: for $d >0$ and $K>1$, we analyze the unstable manifold of the fixed point $(a,a',i,i') = (0,0,K,0)$. It is one-dimensional and has one branch that is asymptotically non-negative, which we call $M^-_d(K)$. In Section \ref{Sec:Asymptotics}, we will prove that any finite segment of $M^-_d(K)$ is continuous both in $d$ and $K$. For $d>0$, this will follow from standard perturbation theory for dynamical systems. For passing from $d = 0$ to $ d \sim 0$, we use geometric singular perturbation theory due to F\'{e}nichel \cite{Jones_Singular_Perturbation}. In Section \ref{Sec:Properties}, we prove that a traveling wave $M^-_d(K)$ persists under small perturbations in $d$, up to its limit as $z \rightarrow + \infty$. The estimate $\iminf + \ipinf = 2 + \mathcal{O}(d)$ is proven in Section \ref{Sec:Mapping_Limits}. We use all previous results to prove the existence of a family of non-negative traveling wave solutions in Section \ref{Sec:Existence_invading_front}. Given the existence of these non-negative traveling waves, we conclude that there also must be an invading front among those.

\subsection{Dynamics around the fixed points}
\label{Sec:Asymptotics}
\subsubsection{A degenerate linearization}
\label{Sec:Degenerate_Linearization}

At a fixed point $(a,a',i,i') = (0,0,K,0)$, the Jacobian of the system has Eigenvalues
\begin{align}
\begin{aligned}
&\lambda_1 = 0, & &\lambda_2 = -\frac{c}{d}, \\ & \lambda_3 = -\frac{c}{2} - \sqrt{ \frac{c^2}{4} + K -1}, && \lambda_4 = - \frac{c}{2} + \sqrt{\frac{c^2}{4} + K -1 }. \label{Eq:Eigenvalues_perturbed}
\end{aligned}
\end{align}
The Eigenvalue $\lambda_2$ is new compared to the unperturbed system, all other Eigenvalues remain the same. The associated Eigenvectors are given by
\begin{align}
\begin{aligned}
e_1 &= 
\begin{pmatrix}
0 \\ 0 \\ 1 \\ 0
\end{pmatrix}, \quad
&&e_2 = \begin{pmatrix}
0 \\ 0 \\ -\frac{d}{c} \\ 1
\end{pmatrix},
	\\
e_3 &= \begin{pmatrix}
- c\lambda_3 - d \cdot \lambda_3^2 \\
-c \cdot \lambda_3^2 - d \cdot \lambda_3^3 \\
K+r \\
\lambda_3 (K+r)
\end{pmatrix},
&& e_4 = \begin{pmatrix}
- c\lambda_4 - d \cdot \lambda_4^2 \\
-c \cdot \lambda_4^2 - d \cdot \lambda_4^3 \\
K+r \\
\lambda_4 (K+r)
\end{pmatrix}. \label{Eq:Eigenvectors_perturbed}
\end{aligned}
\end{align}

For $ K \in [0,1)$, corresponding to a possible limit as $x \rightarrow + \infty$, the Eigenvalues are real-valued if $K \geq 1-c^2/4$. For $K=0$, which is the limit of an invading front, we require that $c \geq 2 $, otherwise converging trajectories can not stay non-negative: the $a$-component spirals around its limit $0$ if $\lambda_3$ and $\lambda_4$ have an imaginary part.

For fixed $d >0$, we first analyze the behavior around a fixed point $(a,a',i,i') = (0,0,K,0)$ locally. The Jacobian of the system at the fixed point is degenerate due to the continuum of fixed points, we apply center manifold theory to work out the higher moments of the approximation. A practical introduction to this topic has been written by J. Carr \cite{Carr_Center_Manifold}. If we ensure that all Eigenvectors \eqref{Eq:Eigenvectors_perturbed} are distinct, they span the entire $\mathds{R}^4$. In this case, the Jacobian can be diagonalized and the calculations are standard: no bifurcaction, neither in $d$ nor $K$, occurs as long as all the Eigenvalues $\lambda_{2,3,4}$ remain real-valued and unequal zero.

If $K=1$, then $\lambda_4 = 0$, and if $K = 1 - c^2/4$, then $\lambda_3 = \lambda_4$ and their eigenspaces become colinear. If we exclude these two cases, the result is as intuitive as it is comfortable: the center manifold locally coincides with the set of fixed points, a defect linearization with three hyperbolic directions and one constant direction is the result, see Proposition \ref{Prop:Linearization_Perturbed}. For the case $d = 0$, we present the rather standard calculations and the necessary changes of coordinates into the system of Eigenvectors in detail \cite{Kreten2022}. The following result is completely analogue:

\begin{proposition} \label{Prop:Linearization_Perturbed}

For $d > 0,c>0$ and $K\in \mathds{R}$ subject to the conditions
\begin{align}
K \neq 1, \quad K \neq 1 - c^2/4,
\end{align}
consider the fixed point $(0,0,K,0)$ of $S_d$. Then, in a neighborhood of the fixed point, the center manifold of the fixed point coincides with the set of points
\begin{align}
\{ a=a'=i'=0, \, i \in \mathds{R} \}.
\end{align}
Moreover, in this neighborhood, the flow $S_d$ is homeomorphic to
\begin{align}
\begin{pmatrix}
	p' \\ u' \\ v ' \\ w'
\end{pmatrix}
 =
\begin{pmatrix}
 	0 \\ \lambda_2 \cdot u \\ \lambda_3 \cdot v \\ \lambda_4 \cdot w
\end{pmatrix}, \label{Eq:Asymptotics_linear}
\end{align}
where $p,u,v,w$ are the coordinates in the system of Eigenvectors \eqref{Eq:Eigenvectors_perturbed}.
\end{proposition}

Proposition \ref{Prop:Linearization_Perturbed} has two important implications. First, regarding the asymptotics as $x \rightarrow + \infty$, we can use Equation \eqref{Eq:Asymptotics_linear} to deduce

\begin{corollary} \label{Cor:Lyapunov_stable}
For $c>0, d >0$ and $K \in (1- \frac{c^2}{4},1) $, the fixed point $(0,0,K,0)$ of $S_d$ is Lyapunov-stable, and trajectories asymptotically converge along $e_{2,3,4}$ \eqref{Eq:Eigenvectors_perturbed}.
\end{corollary}
\begin{proof}
Holds by \eqref{Eq:Asymptotics_linear}, since $\lambda_2, \lambda_3, \lambda_4$ are real-valued and strictly negative.
\end{proof}

\subsubsection{Continuity with respect to the parameters}

Regarding the asymptotics as $x \rightarrow - \infty$, we get
\begin{corollary} \label{Cor:Asymptotic_direction_unstable}
For $d > 0, c>0, K>1$, the fixed point $(0,0,K,0)$ of $S_d$ has a fast unstable manifold of dimension one, associated to the Eigenvalue
\begin{align}
\lambda_4 = -\frac{c}{2} + \sqrt { \frac{c^2}{4} + K -1 } > 0.
\end{align}
This manifold has one branch such that $a,i >0$ asymptotically, which we denote as $M^-_d(K)$. Outgoing from the fixed point, the asymptotic direction of the unstable manifold in coordinates $(a,a',i,i')$ is given by
\begin{align}
\begin{pmatrix}
c\lambda_4 + d \cdot \lambda_4^2 \\
c \cdot \lambda_4^2 + d \cdot \lambda_4^3 \\
-(K+r) \\
-\lambda_4 (K+r)
\end{pmatrix}.
\end{align}
Locally, the manifold $M^-_d(K)$ is continuous in $K$.
\end{corollary}

The local continuity of $M^-_d(K)$ w.r.t $K$ can be extended to arbitrarily finite segments:

\begin{corollary} \label{Cor:Unstable_manifold_linearization}
Let $d > 0, K_0>1, c>0$ and choose any finite time-horizon $T \in \mathds{R}$. Assume that the manifold $M^-_d(K_0,x)\vert_{x \in (- \infty, T]}$ exists, is smooth and bounded. Then, for $K$ sufficiently close to $K_0$, each of the manifolds $M^-_d(K,x)\vert_{ x \in (- \infty, T]}$ is smooth and bounded, and they have a representation that is continuous in $K$.
\end{corollary}
\begin{proof}
The proof is a standard gluing argument: Fix some $K_0 >1$ and a finite time-horizon $T \in \mathds{R}$. Assume that the manifold $M^-_d(K_0,x)\vert_{x \in (- \infty, T]}$ exists, is smooth and bounded. Consider a small ball $B_\delta$ of radius $\delta >0$ around the fixed point $(0,0,K_0,0)$, such that within $B_\delta$, the flow is equivalent to the linearized Flow \eqref{Eq:Asymptotics_linear}. Define the exit time
\begin{align}
x_\delta := \sup_{ x \in \mathds{R} } \big \{ \forall s \leq x : \, M^-_d(K_0,s) \in B_\delta  \big \}, \qquad
x^\ast := \frac{x_\delta}{2}.
\end{align}
As $K \rightarrow K_0$, the trajectories $M^-_d(K,x)\vert_{x \in (- \infty, x^\ast]}$ are continuous in $K$, due to the local statement \ref{Cor:Unstable_manifold_linearization}. In particular, the points $M^-_d(K,x^\ast)$ converge to $M^-_d(K_0,x^\ast)$. We can now treat the rest of the trajectories
\begin{align}
M^-_d(K,x)\vert_{x \in [x^\ast, T]}
\end{align}
as initial value problems with converging initial data. Since $[x^\ast, T]$ is a finite time-interval, this follows from a Gronwall estimate for locally Lipschitz systems, check for example Theorem II-1-2 and Remark II-1-3 in the textbook of Hsieh and Yasutaka \cite{Hsieh_Sibuya_Basic_Theory_ODEs}.
\end{proof}

Similarly, continuity of $M^-_d(K_0,x)\vert_{x \in (- \infty, T]}$ w.r.t. $d$ holds. The result for $d \gg 0$ is standard: local continuity follows from center manifold theory, see e.g. section 1.5 in the monograph of J. Carr \cite{Carr_Center_Manifold}. Is is one of the fundamental tools for analyzing bifurcations, as explained by J. Guckenheimer and P. Holmes, see sections 3.2 and 3.4 in \cite{Guckenheimer_Dynamical}. The assumptions that the Eigenvectors \eqref{Eq:Eigenvalues_perturbed} are distinct and that the Eigenvalues $\lambda_{2,3,4}$ are real-valued and non-zero are again crucial: they imply that locally in $K$ and $d$, no bifurcation occurs. The local statement can easily be extended to arbitrary finite segments as before:

\begin{proposition} \label{Prop:Epsilon_continuity}
Let $K>1, d_0 >0, c >0$. For a finite time-horizeon $T \in \mathds{R}$, assume that the manifold $M^-_{d_0}(K,x) \vert _{ x \in (-\infty,T] }$ is smooth and bounded. There exists an open interval $I \subset \mathds{R}^+_0, d_0 \in I$, such that for all $d \in I$, each of the manifolds $M^-_d(K,x)\vert_{ x \in (- \infty, T]}$ is smooth and bounded, and they have a representation that is continuous in $d$.
\end{proposition}

For passing from $d = 0$ to $d \sim 0$, we use geometric singular perturbation theory. The ODE-system with $d=0$ is three-dimensional, as it is independent of $i'$, but can be embedded into the $\mathds{R}^4$, and then be perturbed smoothly when introducing a small diffusion $d > 0$. The resulting statement is analogue to Proposition \ref{Prop:Epsilon_continuity}, the proof is presented in Appendix \ref{Sec:Geom_sing_appl}:

\begin{corollary} \label{Cor:Cont_finite}
Let $K>1, c >0, r \geq 0$. First consider the fixed point $(\bar{a},\bar{a}',\bar{i}) = (0,0,K)$ of $S_0$, together with its one-dimensional unstable manifold $M^-_0(K)$. Fix any semi-open interval $ x \in (-\infty, T]$, where $T$ is finite, and assume that $M^-_0(K,x)\vert_{x \in (-\infty, T]}$ is smooth and bounded. Lift it naturally into $\mathds{R}^4$ via the fourth coordinate $i'=-a(a+i+r) /c$.

Now consider the perturbed system $S_d$. There exist some $d^\ast > 0$ such that for all $d \in (0,d^\ast)$: the fixed point $(a,a',i,i') = (0,0,K,0)$ has an adjacent one-dimensional unstable manifold $M^-_d(K,x)\vert_{x \in (-\infty, T]}$, that is continuous in $d$ and converges to $M^-_0(K,x)\vert_{x \in (-\infty, T]}$ as $d \rightarrow 0$.
\end{corollary}

\subsection{Persistence of traveling waves under perturbation}
\label{Sec:Properties}

We will not only prove the existence of non-negative traveling waves of $S_d$, but also that they all share certain monotonicity properties. We consider only the case $c \geq 2$, since this is regime in which an invading front can exist.\\

\hypertarget{Ass_A}{\textbf{Properties of traveling waves that are not invading fronts (TW)}}: Let $K >1, d \geq 0, c >0$. Consider the manifold $M^-_{d}(K)$, and denote representing functions $a(x), a'(x), i(x), i'(x) \vert_{ x \in \mathds{R}}$. We say that $M^-_{d}(K)$ admits the traveling wave properties \textbf{(TW)} if the following holds:
\begin{enumerate}
\item $i'(x) < 0 \quad \forall x \in \mathds{R}$,
\item $a(x) >0 \quad \forall x \in \mathds{R}$,
\item $i(x) \geq \ipinf >0 \quad \forall x \in \mathds{R}$,
\item The function $a(x)$ has a unique local maximum, which is also the global one. At the phase-time of the maximum, it holds that $a+i \leq 1$.
\item The trajectory converges monotonously to its limit as $x \rightarrow + \infty$. There exists a finite phase-time $x^\ast$, such that for all $x \geq x^\ast$:
\begin{align}
\begin{aligned}
a'(x) < 0, \quad
i''(x) > 0. \label{Eq:Tail_monotonicity_Ass}
\end{aligned}
\end{align}
\end{enumerate}

An invading front fulfills the same properties, with the exception that $\lim_{x \rightarrow + \infty}i(x)  =0$. However, the assumption $i(x) \geq \gamma >0 $ allows us to perturb the trajectory in such a way that the perturbed solutions stay non-negative. The properties \hyperlink{Ass_A}{\textbf{TW}} have been proven for the non-negative traveling waves of $S_0$, as part of the proof of Theorem \ref{Old_main_Theorem}, see \cite{Kreten2022}. We prove that a given traveling wave persists under small perturbations in $d$:

\begin{proposition} \label{Thm:Non_neg_existance_small_eps}
For $c \geq 2, d_0 \in [0, 3c/2)$ and $K \in (1,2]$, assume that the manifold $M^-_{d_0}(K)$ admits the wave-properties \hyperlink{Ass_A}{\textbf{TW}}. Then, there exists an open interval $I \subset \mathds{R}^+_0, d_0 \in I$, such that for all $d \in I$: the manifold $M^-_d(K)$ also admits \hyperlink{Ass_A}{\textbf{TW}}.
\end{proposition}

The rest of Section \ref{Sec:Properties} is devoted to the proof of Proposition \ref{Thm:Non_neg_existance_small_eps}. The required phase-space analysis is not relevant for the rest of the paper and can be skipped at first reading, we recommend to continue with Section \ref{Sec:Mapping_Limits}.

\subsubsection{Monotonicity, non-negativeness, and an attractor}

Our analysis begins with the fact that $i(x)$ must be monotone as long as the trajectory is non-negative:

\begin{lemma} \label{Lem:I_monotonicity}
Let $K  > 1, c>0$ and $d > 0$. Along the manifold $M^-_d(K,x)$, the inequality
\begin{align}
i'(x) < 0
\end{align}
holds as long as $a(x) >0, i(x) \geq 0$.
\begin{proof}
By Corollary \ref{Cor:Unstable_manifold_linearization}: $i'(x) < 0$ as $x \rightarrow - \infty$. We assume that there exists a finite phase-time $x^\ast$ such that $i'(x^\ast) = 0$ for the first time. This implies that $i(x^\ast) '' \geq 0$, since $i'(x)<0$ for all $x < x^\ast$. However, it also holds that $d i'' = -ci' - ra - a(a+i)$, which implies that $i''(x^\ast) <0$ if $i(x^\ast) \geq 0$ and $a(x^\ast) >0$. Thus, there can not be such a finite time $x^\ast$.
\end{proof}
\end{lemma}

Since we do not change the equation for $a(x)$ in Eq. \eqref{Eq:perturbed_wave}, the following result - that traps $a(x)$ within a non-negative region - can be taken over from the unperturbed system. This statement relies on an analysis of the subsystem for $a(x)$ for fixed $i(x) = K$, and on the fact that $i(x)$ is monotone. We consider only wave-speeds $c \geq 2$, to simplify the representation:

\begin{proposition}[c.f. Thm. 6.1 and Prop. 6.4 in \cite{Kreten2022}] \label{Triangles_3D_Thm}
For $c \geq 2, r \geq 0$ and $d \geq 0$, consider a solution of the Wave-Eq. \eqref{Eq:perturbed_wave_ODE}, that at time $x=0$ is subject to the conditions
\begin{align}
a(0) & >0,\\
 a'(0)& = 0, \\
a(0) + i(0) & \leq 1.
\end{align}
Assume further that there exists some $T \in (0, +\infty]$ such that \begin{align}
i (x) \geq 0, \, i'(x) \leq 0 \quad \forall x \in [0, T).
\end{align}
Then, $a(x)$ is trapped in a non-negative attractor. It holds that
\begin{align}
a(x) > 0 \text{ and } a'(x) < 0 \qquad  \forall x \in (0, T).
\end{align}
Moreover, if for some $\kappa \in (0,1)$ and $x_\kappa \in (0, T)$, it holds that
\begin{align}
a(x_\kappa) + i (x_\kappa) = 1-\kappa,
\end{align}
then there exists a finite constant $L_\kappa \geq 0$, that depends only on $\kappa$, such that
\begin{align}
\int_{s_1}^{s_2} a(x) \, dx &\leq L_\kappa \cdot a(s_1) \quad \forall \, 0 \leq s_1 \leq s_2 \leq T. \label{Ineq:A_integral_bound}
\end{align}
\end{proposition}

Lemma \ref{Lem:I_monotonicity} and Proposition \ref{Triangles_3D_Thm} imply that along $M^-_d(K,x)$, we only have prove that $a(x)$ reaches a local maximum, and that $i(x) \geq 0$ for all $x \in \mathds{R}$, then convergence and non-negativeness follow. As $i'(x)$ essentially depends on $a(x)$, Inequality \eqref{Ineq:A_integral_bound} will be handy for proving that $i'(x) \rightarrow 0$. For $S_0$, we did this in chapters 6.3 and 6.4 of \cite{Kreten2022}. For $S_d$, the new term $di''$ needs to be dealt with, see the following Section \ref{Sec:Tail_convergence}. Before, we conclude this section with another simple phase space argument:

\begin{lemma} \label{Lem:Non_negativeness}
Let $K  > 1, c>0, r \geq 0$ and $d > 0$. Consider the manifold $M^-_d(K,x)$. If there exists some $T \in \mathds{R}$ such that
\begin{align}
i(x) & \geq 0 \quad  \forall x \in (-\infty,T],
\intertext{then $a(x)$ has at most one local maximum, say at some phase-time $x_0 \in (-\infty,T]$. There, it holds that $a(x_0)+i(x_0) \leq 1$. Moreover:}
a(x) & > 0, \, i'(x) < 0 \quad \forall x \in (-\infty,T].
\end{align}
If $T = +\infty$, then both $a(x)$ and $i(x)$ converge and stay non-negative.
\begin{proof}
The manifold $M^-_d(K)$ leaves the fixed point in positive direction of $a$ and negative direction of $i$. We have already proven that $i'(s)<0$ as long as $a>0, i \geq 0$. Assume that $a(s)$ has a first local maximum at some $x_0 \in \mathds{R}$. There, it holds that
\begin{align}
0 \geq a''(x_0) = a(x_0) \cdot \big (a(x_0)+i(x_0) - 1 \big ),
\end{align}
which implies that $a(x_0)+i(x_0)\leq 1$, since $a(x_0) >0$. But since $i(s) \geq 0$ for all $s \in (-\infty,x]$, we can apply Proposition \ref{Triangles_3D_Thm} to trap $a(s)\vert_{x \in (x_0,x]}$: along this part of the trajectory, it holds that $a(x) >0$ and $a'(x) <0 $.
\end{proof}
\end{lemma}

\subsubsection{The tail of a perturbed trajectory}
\label{Sec:Tail_convergence}

For the entire Section \ref{Sec:Tail_convergence}, we work under the assumption that a reference trajectory $M^-_{d_0}(K)$ exists for some $d_0  \geq 0$, which we perturb when changing $d \sim d_0$:\\

\textbf{Assumption}: Let $K>1, c \geq 2, r \geq 0$ and $d_0 \geq 0$. Assume that the manifold $M^-_{d_0}(K)$ admits properties \hyperlink{Ass_A}{\textbf{TW}}, and choose four representing functions
\begin{align}
\bar{a}(x), \bar{a}'(x), \bar{i}(x), \bar{i}'(x), \quad x \in \mathds{R}.
\end{align}

Given $M^-_{d_0}(K)$, we vary the parameter $d$ and track the perturbed trajectories. Therefore, we denote the representing functions of the manifolds $M^-_{d}(K)$ as
\begin{align}
a_d(x), a'_d(x), i_d(x), i'_d(x), \quad x \in \mathds{R},
\end{align}
to emphasize their dependency on $d$.

The results of Proposition \ref{Prop:Epsilon_continuity} and Corollary \ref{Cor:Cont_finite} are structurally similar: they yield continuity of arbitrarily large, but finite time-horizons of $M^-_d(K)$, when varying $d \geq 0$, resulting in the following Proposition \ref{Prop:Finite_non_neg}. It remains to control the tail as $x \rightarrow + \infty$ (for the result, see Prop. \ref{Thm:Non_neg_existance_small_eps}: the wave proverties \hyperlink{Ass_A}{\textbf{TW}} persist for $d$ close to $d_0$).

\begin{proposition} \label{Prop:Finite_non_neg}
Let the manifold $M^-_{d_0}(K)$ be as above and choose a finite time-horizon $T \in \mathds{R}$. For all $\epsilon>0$, there exists an open interval $I \subset \mathds{R}^+_0, d_0 \in I$, such that for all $d \in I$: the manifold $M^-_d(K,x)\vert_{x \in (-\infty,T]}$ is of distance at most $\epsilon$ to $M^-_{d_0}(K,x)\vert_{x \in (-\infty,T]}$ and is strictly non-negative.
\begin{proof}
For the reference trajectory, note that there exists a $\gamma >0$ such that
\begin{align}
\bar{i}(x) \geq \gamma \quad \forall x \in (-\infty,T].
\end{align}
If $d_0 = 0$, apply Corollary \ref{Cor:Cont_finite}, if $d_0>0$ apply Proposition \ref{Prop:Epsilon_continuity}, both yield continuity in $d$: for all $\epsilon>0$, there exists an open interval $I \subset \mathds{R}^+_0, d_0 \in I$ such that for all $d \in I$:
\begin{align}
\big \vert \big \vert M^-_d(K,x) - M^-_{d_0}(K,x) \big \vert \big \vert_\infty & \leq \epsilon \quad &&\forall x \in (-\infty,T].
\intertext{In particular, for $\epsilon \leq \gamma /2 $:}
i_d(x) \geq \bar{i}(x) - \epsilon &\geq \gamma - \gamma /2  > 0 \quad &&\forall x \in (-\infty,T].
\end{align}
By Lemma \ref{Lem:Non_negativeness}, this implies also positiveness of $a_d(x)\vert_{x \in (-\infty,T]}$.
\end{proof}
\end{proposition}

We trap $M^-_d(K)$ in the attractor from Proposition \ref{Triangles_3D_Thm}:

\begin{lemma} \label{Lem:A_phase_maximum}
Let $M^-_{d_0}(K)$ as before. There exists a constant $\kappa \in (0,1)$, an open interval $I_0 \subset \mathds{R}^+_0, d_0 \in I_0$, and a finite phase-time $x_\kappa$, such that for all $d \in I_0$:
\begin{enumerate}
\item the unstable manifold $M^-_d(K,x)\vert_{x \in (-\infty,x_\kappa]}$ is strictly positive and has a unique first local maximum of active particles at a finite phase-time $\tilde{x}_0(d) \in (-\infty, x_\kappa]$.
\item It holds that
\begin{align*}
a_d'(x_\kappa) < 0, \quad
a_d(x_\kappa) + i_d(x_\kappa) &\leq 1 - \kappa.
\end{align*}
\end{enumerate}

\begin{proof}
On $M^-_{d_0}(K,x)$, there exists a unique sharp global maximum of $\bar{a}(x)$, say at some phase-time $x_0$. Choose a phase-time $x_\kappa > \tilde{x}$ so large, such that
\begin{align}
\bar{a}(x_\kappa) + \bar{i}(x_\kappa) & \leq 1 - 2 \kappa, \quad \text{ for some $\kappa \in (0,1)$}, \\
\bar{a}'(x_\kappa) & \leq 2 \delta < 0, \quad \text{ for some $\delta >0$,}
\end{align}
which must exist since $M^-_{d_0}(K,x)$ admits \hyperlink{Ass_A}{\textbf{TW}}.

Apply the previous Proposition \ref{Prop:Finite_non_neg} over the interval $(-\infty, x_\kappa]$, such that we can control $M^-_d(K,x)\vert_{x \in (-\infty, x_\kappa]}$ for $d$ in some open interval $I \subset \mathds{R}^+_0$. We choose $\vert d-d_0 \vert$ sufficiently small such that
\begin{align}
i_d(x) & > 0 \quad \forall x \in (- \infty, x_\kappa], \\
a_d(x_\kappa) + i_d(x_\kappa) & \leq 1 - \kappa, \\
a_d'(x_\kappa) & \leq \delta < 0.
\end{align}
By Corollary \ref{Cor:Unstable_manifold_linearization}: $a_d'(x) >0$ as $x \rightarrow - \infty$. Hence, also $a_d(x)$ must have a local maximum before $x_\kappa$. It is unique since $i_d(x)\vert_{x \in (-\infty,x_\kappa]} > 0$ and by Lemma \ref{Lem:Non_negativeness}.
\end{proof}
\end{lemma}

In view of Proposition \ref{Triangles_3D_Thm}, we now have trapped $M^-_d$ in a non-negative monotone attractor, as long as we can control $i_d \geq 0$. We first prove monotonicity of the tail as $x \rightarrow + \infty$:

\begin{lemma} \label{Lem:Ip_monotonicity}
Consider $M^-_{d_0}(K), I_0$ and $x_0$ as in Lemma \ref{Lem:A_phase_maximum}. There exists a phase-time $x_{tail} \geq x_\kappa$, such that for all $x^\ast \in ( x_{tail}, \infty)$:

There exists an open interval $I_1 \subset I_0, d_0 \in I_1$, such that for all $d \in I_1$:

\begin{align}
a_d(x^\ast),i_d(x^\ast) &>0, \nonumber  \\
 a_d'(x^\ast), i_d'(x^\ast) &< 0, \label{Eq:Monotonicity_tail_full} \\
 i_d''(x^\ast) &>0. \nonumber
\intertext{Moreover, given the existence of such an $x^\ast$, it holds for all $x \geq x^\ast$, and as long as $i_d(x) \geq 0$:}
i_d''(x) & > 0,
\end{align}
and thus all $a_d, a'_d,i_d,i'_d$ converge monotonously to zero.
\begin{proof}
The Inequalities \eqref{Eq:Monotonicity_tail_full} are true on the tail of $M^-_{d_0}$, say for all $x \geq x_1$, where we let $x_1 \geq x_\kappa$ without loss of generality. Now pick some $x_2 > x_1$. There exists an open interval $I_1 \subset I_0, d_0 \in I_1$, such that for all $d \in I_1$:

The manifold $M^-_d(K,x)\vert_{x \in (-\infty,x_2]}$ exists, is non-negative and converges to $M^-_0(K,x)\vert_{x \in (-\infty,x_2]}$ as $d \rightarrow d_0$. By continuity: for $\vert d -d_0 \vert$ sufficiently small and for all $x \in [x_1, x_2]$:
\begin{align}
a_d(x),i_d(x) >0, a'_d(x) <0, i'_d(x) <0.
\end{align}
On $[x_1, x_2]$: $\bar{i}''(x) > 0$, so $\bar{i}'(x_1) < \bar{i}'(x_2)$. Since $i'_d \rightarrow \bar{i}'$, there must also be some $x^\ast_d \in [x_1, x_2]$, such that $i_d''(x^\ast_d) >0$, for all $d$ such that $\vert d - d_0 \vert$ is sufficiently small.

We assume that there exists a finite time $x_3 \geq x^\ast_d$ such that $i''_d(x_3) = 0$ for the first time after $x^\ast$. This implies that $i'''_d(x_3) \leq 0$. On the other hand, deriving the equation for $i$ in \eqref{Eq:perturbed_wave} once yields
\begin{align}
\begin{aligned}
d \cdot i''' &= -(ci' + ra + a(a+i))' \\
&= -\big(ci'' + ra' + a'(a+i) + a(a'+i') \big) \\
&= -ra' - a'(a+i) - a(a'+i').
\end{aligned}
\end{align}
As long as $i_d \geq 0$, we can apply Lemmas \ref{Lem:Non_negativeness} and \ref{Lem:I_monotonicity}. Then, all three terms in the last line are strictly positive, resulting in $i_d'''(x_3) > 0$, contradicting our assumption. This allows to fix $x_{tail} = x_2$, independent of $d$.
\end{proof}
\end{lemma}

\textit{Remark}: Note that we can choose $x^\ast$ from Lemma \ref{Lem:Ip_monotonicity} as large as we want, the price is a stronger restriction regarding $d \sim d_0$. This will be helpful if we assume that $M^-_{d_0}(K)$ converges: also the perturbed trajectories get as close to the limit as we need.\\

Given monotonicity of the tail, we can control $i_d(x)$ for $x \in [x^\ast, \infty)$:

\begin{corollary} \label{Cor:Ip_monotonicity}
Consider $M^-_{d_0}(K)$ and $I_1,x^\ast$ as in Lemma \ref{Lem:Ip_monotonicity}. There exists a finite constant $J \geq 0$, such that for all $d \in I_1$ and all phase-times $x_2 \geq x_1 \geq x^\ast$, the bound
\begin{align}
\begin{aligned}
c \cdot \big \vert \big ( i_d(x_1) - i_d(x_2) \big) \big \vert  \leq  d \cdot \vert i_d'(x_1) \vert + J \cdot a_d(x_1)
\end{aligned} \label{Eq:I_bound_by_tail_rest}
\end{align}
holds as long as $i_d(x) \geq 0$ for all $x \in ( - \infty, x_2]$.

\begin{proof}
Choose $x^\ast$ as in the previous Lemma \ref{Lem:Ip_monotonicity}. For $x_2 \geq x_1 \geq x^\ast$, integrate $0 = d i_d'' + c i_d' + a_d(a_d+i_d+r)$:
\begin{align}
c \cdot \big \vert  i_d(x_2) - i_d(x_1) \big \vert &= \big \vert  d \cdot \big (i_d'(x_2) - i_d'(x_1) \big )  + \int_{x_1}^{x_2} a_d(a_d+i_d+r) \, ds \,  \big \vert.
\intertext{As long as $i_d \geq 0$, Proposition \ref{Triangles_3D_Thm} and the previous Lemmas \ref{Lem:Non_negativeness}, \ref{Lem:I_monotonicity} and \ref{Lem:Ip_monotonicity} imply $a_d \geq 0 , a_d'\leq 0, i_d' \leq 0$ and $i_d'' \leq 0$. Thus, since $a_d+i_d+r \leq 1+r$:}
c \cdot \big \vert  i_d(x_2) - i_d(x_1) \big \vert & \leq
d \cdot \big \vert i_d'(x_1) \big \vert + (1+r) \cdot \int_{x_1}^{x_2} a_d(x) \, dx.
\intertext{By our bound $a_d(x_\kappa) + i_d(x_\kappa) \leq 1 - \kappa$, for some $x_\kappa \leq x^\ast$, Proposition \ref{Triangles_3D_Thm} implies that there exists a constant $L_\kappa \geq 0$ such that}
\int_{x_1}^{x_2} a(s) \, ds & \leq L_\kappa \cdot a_d(x_1) \qquad \forall x_2 \geq x_1 \geq x_0,
\end{align}
as long as $i_d(x) \geq 0$ for all $x \in (-\infty, x_2]$.
\end{proof}
\end{corollary}

We now prove the existence of a non-negative traveling wave for $d \sim d_0$, all that is left to be done is to prove convergence as $x \rightarrow + \infty$:

\begin{proof}[\textbf{Proof of Proposition \ref{Thm:Non_neg_existance_small_eps}}]

Consider a manifold $M^-_{d_0}(K)$ that fulfills \hyperlink{Ass_A}{\textbf{TW}}. Then, there exists a constant $\gamma >0$, such that $\bar{i}(x) \geq \gamma$ for all $x \in \mathds{R}$. We prove that for all $d$ sufficiently close to $d_0$:
\begin{align}
\vert i_d (x) - \bar{i}(x) \vert \leq \frac{\gamma}{2} && \forall x \in \mathds{R}.
\end{align}
This implies that $i_d(x) \geq \frac{\gamma}{2} >0$, and Lemma \ref{Lem:I_monotonicity} yields positiveness and convergence of $M^-_d(K)$. The rest of the wave properties \hyperlink{Ass_A}{\textbf{TW}} then follows by Lemma \ref{Lem:I_monotonicity}.

Let $x^\ast$ as in Corollary \ref{Cor:Ip_monotonicity}: there exists an open interval $ I_1 \subset \mathds{R}_0^+$ such that for all $d \in I_1$: the trajectory $M^-_d(K,x) \vert_{x \in (-\infty,x^\ast]}$ is non-negative, and there exists some $J \geq 0$ such that for all $x_2 \geq x_1 \geq x^\ast$:
\begin{align}
i_d(x_2) \geq i_d(x_1) - \frac{d}{c} \vert i'_d(x_1) \vert - \frac{J}{c} \cdot a_d(x_1), \label{Eq:Tail_bound_final_proof}
\end{align}
as long as $i_d \geq 0$. Now, choose $x_1 \geq x^\ast$ large enough such that both
\begin{align}
\frac{J}{c} \cdot \bar{a}(x_1) & \leq \delta : = \frac{\gamma}{7}, \label{Eq:Tail_bound_unp1} \\
 \vert\bar{a}(x_1)\vert + \vert\bar{a}'(x_1)\vert + \vert\bar{i}'(x_1)\vert + \vert \bar{i}(x_1) - \gamma\vert & \leq \delta : = \frac{\gamma}{7}, \label{Eq:Tail_bound_unp2}
\end{align}
which is possible since $M^-_{d_0}(K)$ converges. Then choose an even smaller open interval $I_2 \subset I_1, d_0 \in I_2$, such that by continuity: for all $d \in I_2$, $M^-_d(K,x)\vert_{x \in (- \infty, x_1]}$ is non-negative and of distance at most
\begin{align}
\epsilon = \frac{\gamma}{7(1+\frac{J}{C})}  \intertext{to $M^-_{d_0}(K,x)\vert_{x \in (- \infty, x_1]}$. Then, in view of Eq. \eqref{Eq:Tail_bound_unp1} and \eqref{Eq:Tail_bound_unp2}, also}
	\frac{J}{c} \cdot a_d(x_1) &\leq \frac{2 \gamma}{7} = 2 \delta, \\
\vert a_d(x_1) \vert + \vert a'_d(x_1) \vert + \vert i'_d(x_1) \vert & \leq \frac{ 2 \gamma}{7} = 2 \delta. \label{Eq:Tail_perturbed_small}
\end{align}
This implies that for all $d \in I_2$ and $x_2 \geq x_1$,  and by Eq. \eqref{Eq:Tail_bound_final_proof}:
\begin{align}
\begin{aligned}
i_d(x_2) & \geq i_d(x_1) - \frac{d}{c} \Big \vert i'_d(x_1) \Big \vert - \frac{J}{c} \cdot a_d(x_1) \\
& \geq \bar{i}(x_1) - \epsilon - \frac{d}{c} \Big \vert \bar{i}'(x_1) + \epsilon \Big \vert - 2 \delta \\
& \geq \gamma - \frac{\gamma}{7} -  \frac{d}{c} \cdot \Big \vert \frac{\gamma}{7} + \frac{\gamma}{7} \Big \vert - 2 \frac{\gamma}{7} \\
& \geq \frac{ 4\gamma}{7} - \frac{ d }{c} \cdot \frac{2\gamma}{7}. \label{Eq:Positivity_proof_final_bound}
\end{aligned}
\end{align}
The last line is strictly positive for $d < 3c/2$. Thus, $i_d(x_2) \geq 0$ for all $x_2 \geq x_1$, and the trajectory stays non-negative and converges monotonously to its limit, as claimed.
\end{proof}

\subsection{The mapping of the limits}
\label{Sec:Mapping_Limits}
In view of the previous Section, the wave properties \hyperlink{Ass_A}{\textbf{TW}} persist under small perturbations in $d$ as long as $\ipinf > 0$. However, we can relate the limits of any bounded and non-negative solution up to a term of order $\mathcal{O}(d)$:

\begin{proposition} \label{Lem:Epsilon_limit_bounds}
For $d > 0$, the two limits $\iminf$ and $\ipinf$ of any smooth, bounded and non-negative traveling wave (where $a \not \equiv 0$) fulfill
\begin{align}
2 - d \cdot \frac{2(r+1)}{c^2} < \iminf + \ipinf < 2 . \label{Eq:bound_iminf_ipinf}
\end{align}
\end{proposition}

For the proof, we first verify integrability of any non-negative traveling wave:

\begin{lemma} \label{Lem:global_integrability_perturbed}
For $d > 0$, let $a(x),a'(x),i(x),i'(x)\vert_{ x \in \mathds{R}}$ be a smooth, bounded and non-negative solution of System \eqref{Eq:perturbed_wave_ODE}, where $a \not\equiv 0$. Then, as $x \rightarrow \pm \infty$, $a(x)$ vanishes and $i(x)$ converges, and $a,a',a'',i',i'' \in L^1(\mathds{R})$. Moreover, $i(x)$ is not constant and $i' \leq 0$. The function $a(x)$ has a unique global and local maximum, and $a(x) >0$ for all $x \in \mathds{R}$.

\begin{proof}
Let $a(x),b(x),i(x),i'(x)$ as above. If $i(x)$ does not converge at either $x \rightarrow + \infty$ or $x \rightarrow - \infty$, it must oscillate: in this setting, we can use the reasoning in the proof of Lemma \ref{Lem:I_monotonicity} to provoke a contradiction. Thus, $i(x)$ is either increasing or decreasing. If $i'(x) \geq 0$ for all $x \in \mathds{R}$, then
\begin{align}
- d \cdot i'' \geq i' + a(a+i+r) \geq 0.
\end{align}
As a consequence, the trajectory can not be bounded, since $i'(x)$ can not vanish at both $x \rightarrow \pm \infty$, or $i(x)$ must be constant. However, if $i(x)$ was constant, then
\begin{align}
0 \equiv d i'' + ci' = -a(a+i+r),
\end{align}
which can not hold since we assumed that $a \not \equiv 0$. It follows that $i(x)$ is decreasing and by boundedness must converge as $\ta \rightarrow \pm \infty$, so $i' \in L^1(\mathds{R})$. The two limits must be fixed points, we denote them as $(0,0,\iminf,0)$ and $(0,0,\ipinf,0)$.

Assume that at some finite phase-time $x^\ast$: $a(x^\ast) = 0$. Since we assumed that $a \geq 0$, this must be a local minimum, so $a'(x^\ast) = 0$. Then also $-a''(x^\ast) = ca'(x^\ast) + a(x^\ast) -a(x^\ast)(a(x^\ast)+i(x^\ast)) = 0$, and by induction: $a^{(n)}(x^\ast) = 0$ for all $n \in \mathds{N}$. But this contradicts the assumption $a \not\equiv 0$.

For $a(\ta) \not\equiv 0$, there is at least one local maximum of active particles, since the trajectory converges at both ends. We denote this maximum as $(a_0,0,i_0,i'_0)$. At this point, $a'' = a_0 (a_0 + i_0 -1) \leq 0$, so $a_0 + i_0 \leq 1$, since $a_0 >0$. Assume that there is also a local minimum of $a(\ta)$, denoted as $(a_m,0,i_m,i_m')$. Since $a(\ta)$ vanishes at both ends, we may assume without loss of generality that this be the first local minimum after passing through $(a_0,0,i_0,i_0')$.

Now, since $a_m(a_m + i_m-1-1) =a_m'' \geq 0$, it must hold that $a_m + i_m \geq 1$. But $i(\ta)$ is decreasing, so $a(\ta)$ must have been increasing, a contradiction to the assumption that this is the first local minimum after the maximum $(a_0,0,i_0,i_0')$. As a consequence, there is only one local maximum of active particles, which is also the global one. Further, this implies $a' \in L^1(\mathds{R})$.

Given that $i' \in L^1(\mathds{R})$, we know that the expression
\begin{align}
\int_{\mathds{R}}  d \cdot i(\ta)'' + a(\ta)\big[a(\ta)+i(\ta)+r\big]   dx = c  ( \iminf - \ipinf) \label{Eq:I_int_1}
\end{align}
is finite, however the left side might only exist in a Riemannian sense. Integrating the left-hand side over a finite interval $[-M,M]$, it also holds that
\begin{align}
\begin{aligned}
&\int_{-M}^{M} d \cdot i(\ta)'' + a(\ta)\big[a(\ta)+i(\ta)+r\big] dx \\ &= d  \big [ i'(M) - i'(-M) \big ] + \int_{-M}^{M} a(\ta) \big [a(\ta)+i(\ta)+r\big]  dx. \label{Eq:I_int_2}
\end{aligned}
\end{align}
By monotonicity of $i(x)$, it holds that $i'(x) \rightarrow 0$ as $x \rightarrow \pm \infty$. By \eqref{Eq:I_int_1} and \eqref{Eq:I_int_2}:
\begin{align}
\begin{aligned}
c  ( \iminf - \ipinf) = \int_{\mathds{R}} a(\ta) \big [a(\ta)+i(\ta)+r \big ] dx,
\end{aligned}
\end{align}
and the right-hand side is integrable, since all terms are non-negative. In view of \eqref{Eq:I_int_2}, also $i''(x)$ is integrable, as a sum of integrable terms. We proceed in a similar fashion with the equation $a''+ca' +a = a(a+i)$. We integrate it over the finite interval $[-M,M]$:
\begin{align}
	  \int_{-M}^{M} a''(\ta) + c a'(\ta) + a(\ta) \, d \ta 
	 =   \int_{-M}^{M} a(\ta) \cdot \big[ a(\ta) + i(\ta) \big] d \ta.
\end{align}
We already know that the right-hand is integrable, and that both $a(\pm M)$ and $a'(\pm M)$ vanish as $M \rightarrow +\infty$. This implies
\begin{align}
\begin{aligned}
 & \int_{\mathds{R}} a(\ta) d\ta \\
 & =
\lim_{M \rightarrow + \infty} \Big [ a'(M) -a'(-M) + c  \big [a(M) - a(-M) \big] + \int_{-M}^{M} a(\ta)  d \ta \Big ] \\ 
	  & = \int_{\mathds{R}} a(\ta) \cdot \big[ a(\ta)+i(\ta) \big] \, d\ta .
	  \end{aligned}
\end{align}
Hence also $a(x)$ is integrable, since it is non-negative. Finally, as a sum of integrable terms, also $a''(x)$ is integrable.
\end{proof}
\end{lemma}

\begin{proof}[\textbf{Proof of Proposition \ref{Lem:Epsilon_limit_bounds}}]
Given integrability a non-negative traveling wave, we define
\begin{align}
A (x) := \int_{- \infty}^x a(s), \quad \mathcal{A}:= A( + \infty).
\end{align}
Then, integrating the equation for $a(x)$ in \eqref{Eq:perturbed_wave}, it must hold that
\begin{align}
\begin{aligned}
\mathcal{A} &= \int_\mathds{R} a(\ta)\big[a(\ta)+i(\ta)\big] dx >0, \\
c \cdot \big ( \iminf - \ipinf \big ) &= r \cdot \mathcal{A} + \int_{ \mathds{R} } a(\ta)\big[a(\ta)+i(\ta)\big] dx.
\end{aligned} \label{Eq:relation_part_1}
\end{align}
Moreover, by integration by parts:
\begin{align}
\begin{aligned}
& \int_{\mathds{R}} a(\ta)\big[a(\ta)+i(\ta)\big] dx \\
& = \mathcal{A} \cdot \ipinf - \int_{\mathds{R}} A(x) \big [ a'(x)+i'(x) \big  ]dx \\ 
&= \mathcal{A} \cdot \ipinf + \frac{1}{c} \int_{\mathds{R}}
A(x) \big [ (1+r) a(x) + d i''(x) + a''(x) \big ] dx\\
&= \mathcal{A} \cdot \ipinf + \frac{1+r}{2c} \mathcal{A}^2
+ \frac{1}{c} \int_{\mathds{R}} A(x) \cdot a''(x) \, dx
+ \frac{d}{c} \int_{\mathds{R}} A(x) \cdot i''(x) \, dx\\
&= \mathcal{A} \cdot \ipinf + \frac{1+r}{2c} \mathcal{A}^2 - \frac{ d}{c} \int_{\mathds{R}} a(x) \cdot i'(x) \, dx.
\end{aligned}
\end{align}
Using this and Eq. \eqref{Eq:relation_part_1}, it follows that
\begin{align}
\mathcal{A} &= \frac{c}{1+r} \cdot ( \iminf - \ipinf ), \label{Eq:I_bound_last_step1} \\
\mathcal{A} &= \mathcal{A} \cdot \ipinf + \frac{1+r}{2c} \mathcal{A}^2 - \frac{d}{c} \cdot \int_{\mathds{R}} a(x) \cdot i'(x) \, dx. \label{Eq:I_bound_last_step2}
\end{align}
We know that $ \int_{\mathds{R}} a \cdot i' < 0$. We drop this term in Eq. \eqref{Eq:I_bound_last_step2}, divide by $\mathcal{A} >0$ and then eliminate the variable $\mathcal{A}$ via \eqref{Eq:I_bound_last_step1}, leading to 
\begin{align}
1 > \ipinf + \frac{1}{2}(\iminf - \ipinf).
\end{align}
The first part of the claim follows: $2 > \ipinf + \iminf$.

At the local maximum of $a(x)$, it holds that $a < a + i \leq 1$, see the proof of Lemma \ref{Lem:Non_negativeness}. Thus $-i' \geq -ai' \geq  0$, and we can bound (using  \eqref{Eq:I_bound_last_step1}):
\begin{align}
-\frac{d}{c} \int_{\mathds{R}} a(x) \cdot i'(x) \, dx < \frac{d}{c} \cdot ( \iminf - \ipinf) = \frac{d}{c} \cdot \frac{1+r}{c} \mathcal{A}.
\end{align}
As before, we apply this bound to \eqref{Eq:I_bound_last_step2}, divide by $\mathcal{A}$ and eliminate $\mathcal{A}$ via \eqref{Eq:I_bound_last_step1}, which leads to
\begin{align}
2 < \iminf + \ipinf + 2\frac{d (1+r)}{c^2}.
\end{align}
\end{proof}

\textit{Remark}: These estimates seem to be sharp only asymptotically for $d \sim 0$, see the numerical Table \ref{Tab:Root_iminf} in Appendix \ref{Sec:Numerical_spectrum}.

\subsection{Existence of the traveling waves}
\label{Sec:Existence_invading_front}

The existence of an invading front is proven in two steps: we first prove that for $d >0$, there exists a traveling wave that is non-negative, but not necessarily a traveling front, as long as we can control the limits via the $\mathcal{O}(d)$-Estimate \eqref{Eq:bound_iminf_ipinf}. Afterwards, we show that for fixed $d \geq 0$, the existence of an arbitrary non-negative traveling wave also implies the existence of an invading front, by using the continuum of possible limits.

\subsubsection{Arbitrary bounded non-negative solutions}
\begin{theorem}[Existence of a traveling wave solution] \label{Prop:Existence_of_a_tr_sol}
Let $r \geq 0, c \geq 2$ and $\iminf \in (1,2)$. Then, for all $d > 0$ that fulfill
\begin{align}
d < \frac{3c}{2}, \label{Eq:Eps_bound_1_thm} \\
d \frac{2(r+1)}{c^2} & \leq 2 - \iminf, \label{Eq:Eps_bound_2_thm}
\end{align}
there exists a non-negative traveling wave with limits
\begin{align}
\lim_{x \rightarrow \pm \infty} a(x) &= 0, \\
\lim_{x \rightarrow - \infty} i(x) &= \iminf, \\
\lim_{x \rightarrow + \infty} i(x) & > 2 - \iminf - d \frac{2(r+1)}{c^2} \geq 0. \label{Eq:Est_Existence_plus_infty_bound}
\end{align}
All these waves admit the wave properties \hyperlink{Ass_A}{\textbf{TW}}. Moreover, for $d \in (0,1)$, the waves converge exponentially fast. The rates of convergence depend on $i_{\pm \infty}$ and are given by \eqref{Eq:Rates_Conv_Thm_old}, as in the case $d=0$.
\begin{proof}
Fix $\iminf \in (1,2)$, and consider the manifold $M^-_d(\iminf)$ as before. We prove \hyperlink{Ass_A}{\textbf{TW}} for all $d$ as claimed. For $d = 0$, this statement is part of the proof of Theorem \ref{Old_main_Theorem}. We will use continuity in $d$, see Proposition \ref{Thm:Non_neg_existance_small_eps}, for which we need Bound \eqref{Eq:Eps_bound_1_thm}. The second Bound \eqref{Eq:Eps_bound_2_thm} is the one from Proposition \ref{Lem:Epsilon_limit_bounds}, it will ensure that $\lim_{x \rightarrow + \infty} i(x) > 0$. The rates of convergence will be discussed when finishing the proof of Theorem \ref{Thm:existence_small_epsilon}, at the end of Section \ref{Sec:Invading_Front}.

Since \hyperlink{Ass_A}{\textbf{TW}} is true for $d = 0$, it is also true for $d \in [0, d_1)$, for some $d_1 >0$, by Proposition \ref{Thm:Non_neg_existance_small_eps}. Let
\begin{align}
d^\ast = \sup_{d_1 \geq 0} \big \{\forall d \in [0, d_1): \,  \text{\hyperlink{Ass_A}{\textbf{TW}} holds for } M^-_d(\iminf)  \big \} >0.
\end{align}
Assume that $d^\ast$ does violate neither \eqref{Eq:Eps_bound_1_thm} nor \eqref{Eq:Eps_bound_2_thm}. The manifold $M^-_{d^\ast}(\iminf)$ exists locally around the fixed point $(a,a',i,i') = (0,0,\iminf,0)$ due to Corollary \ref{Cor:Asymptotic_direction_unstable}. Let $R >0$. Since the System $S_{d^\ast}$ is a smooth vector field, also the continuation
\begin{align}
M^-_{d^\ast}(\iminf) \cap \{ x \in \mathds{R}^4: \vert\vert x \vert\vert_\infty < R\}
\end{align}
is a smooth and bounded manifold. If it is non-negative, we proceed with the next paragraph. If not, choose a representation 
$$M^-_{ d^\ast }(\iminf,x) \vert_{x \in (- \infty,T]}$$
for some finite $T \in \mathds{R}$, such that the manifold has become strictly negative at time $T$. By Proposition \ref{Prop:Epsilon_continuity}, there exists an open interval $I \subset \mathds{R}^+_0, d^\ast \in I$, such that for all $d \in I$: the manifolds $M^-_{ d}(\iminf,x) \vert_{x \in (- \infty,T]}$ exists and are continuous in $d$. However, for all $d < d^\ast$, these manifolds are also non-negative, by choice of $d^\ast$. Thus, also $M^-_{ d^\ast }(\iminf,x) \vert_{x \in (- \infty,T]}$ must be non-negative.

Since $R>0$ was arbitrary, the entire manifold $M^-_{ d^\ast}(\iminf)$ must be non-negative. Since it is not constant, it fulfills the conditions of Proposition \ref{Lem:Epsilon_limit_bounds}, so
\begin{align}
\lim_{x \rightarrow + \infty} i_{ d^\ast}(x) > 2 - \iminf - d^\ast \frac{2(r+1)}{c^2} \geq 0,
\end{align}
in view of Condition \eqref{Eq:Eps_bound_2_thm}, which was not violated. Now, Lemma \ref{Lem:global_integrability_perturbed} yields some structural results: it holds that $i'_{ d^\ast}(x) < 0, a_{ d^\ast}(x) >0$, and $a_{ d^\ast}(x)$ has a unique local and global maximum. Since $a_{ d^\ast}(x), i_{ d^\ast}(x)$ converge as $x \rightarrow + \infty$, there exists some $x^\ast$ as claimed in Lemma \ref{Lem:Ip_monotonicity}, which yields monotonicity of the tail. Thus, we have verified \hyperlink{Ass_A}{\textbf{TW}} also for $ d^\ast$, and can again apply Proposition \ref{Thm:Non_neg_existance_small_eps} to $M^-_{ d^\ast}$. This proves \hyperlink{Ass_A}{\textbf{TW}} in a local neighborhood of $ d^\ast$. Insofar, $ d^\ast$ can not have been chosen as claimed, and either \eqref{Eq:Eps_bound_1_thm} or \eqref{Eq:Eps_bound_2_thm} must be violated for $d^\ast$.
\end{proof}
\end{theorem}

By choosing $\iminf$ arbitrarily close to $1$, Theorem \ref{Prop:Existence_of_a_tr_sol} has a simple
\begin{corollary} \label{Cor:Existence_tr_wave}
Let $r \geq 0, c \geq 2$. Then, for all $d \geq 0$ such that
\begin{align}
d  < \min \big \{ \frac{3c}{2} , \frac{c^2}{2(r+1)} \big \},
\end{align}
there exists some $\iminf \in (1,2)$ such that the manifold $M^-_d(\iminf)$ is a non-negative traveling wave.
\end{corollary}

\subsubsection{The invading front}
\label{Sec:Invading_Front}

It remains to show that there exists a wave where $\ipinf = 0$, despite the $\mathcal{O}(d)$-estimate regarding the limits \eqref{Eq:Est_Existence_plus_infty_bound}. We begin with the simple

\begin{lemma} \label{Lem:I_must_become_negative}
For all $d > 0, c>0, r \geq 0$, the manifold
\begin{align}
M^-_d (2)
\end{align}
is not non-negative.
\begin{proof}
Asymptotically as $x \rightarrow - \infty$: $a(x) >0$ by Corollary \ref{Cor:Asymptotic_direction_unstable}. Assume that $a(x),i(x) \geq 0$ for all $x \in \mathds{R}$. Then, denoting the other limit of the wave as $\ipinf = \lim_{x \rightarrow + \infty} i(x) \geq 0$, Proposition \ref{Lem:Epsilon_limit_bounds} implies
\begin{align}
2 + \ipinf <2,
\end{align}
a contradiction.
\end{proof}
\end{lemma}

Given a non-negative and converging trajectory $M^-_d(K_0)$, the following lemma yields continuity and non-negativeness of $M^-_d(K)$ with respect to $K$ in a neighborhood of $K_0$:

\begin{lemma} \label{Lem:K_continuity}
Let $c \geq 2, d > 0, r \geq 0$ and $K_0 \in (1,2)$. Assume that the manifold $M^-_d(K_0)$ is a non-negative traveling wave, and thus converges as $x \rightarrow + \infty$, where
\begin{align}
\lim_{ x \rightarrow + \infty} i(x) = \ipinf \in (0,1).
\end{align}
Then, there exists an open interval $I \subset (1,2), K_0 \in I$ such that for all $K \in I$: the entire manifold $M^-_d(K)\vert _{ x \in \mathds{R}}$ is continuous in $K$, and moreover is also a non-negative traveling wave.

\begin{proof}
Consider $M^-_d(K_0)$ as proposed, and denote its limit as $\ipinf^\ast$. We use Corollary \ref{Cor:Unstable_manifold_linearization} to control $M^-_d(K,x) \vert_{x \in ( - \infty, T]}$ for some $T \in \mathds{R}$. By Corollary \ref{Cor:Lyapunov_stable}, the fixed point $(0,0,\ipinf^\ast,0)$ is Lyapunov stable, and we can control also the tail as $x \rightarrow + \infty$.

Choose some
\begin{align}
\rho \in (0, \ipinf^\ast) .
\end{align}
There exists $ \epsilon >0$, such that for all trajectories $a(x), a'(x), i(x), i'(x)$ that start with in range of $\epsilon$ to $(0,0,\ipinf^\ast,0)$:
\begin{align}
\Big \vert \big ( a(x), a'(x), i(x), i'(x) \big ) - \big (0,0,\ipinf^\ast,0 \big) \Big \vert \leq \rho, \quad \text{for all $x \geq 0$},
\end{align}
since the fixed point is Lyapunov-stable. Wlog let $\epsilon \leq \rho$. There exists a finite phase-time $T$, such that
\begin{align}
\Big \vert M^-_d(K_0,T) - \big (0,0,\ipinf,0 \big) \Big \vert \leq \frac{ \epsilon }{2}.
\end{align}
We now use the continuity of $M^-_d(K,x) \vert _{ x \in (-\infty,T]}$  with respect to $K$, see Corollary \ref{Cor:Unstable_manifold_linearization}: there exists a $\delta >0$ with $K_0 - \delta > 1$, and such that for all $K \in [K_0 - \delta, K_0 + \delta]$: the manifold $M^-_d(K,x) \vert _{ x \in (-\infty,T]}$ is of distance at most $\epsilon / 2$ to $M^-_d(K_0,x) \vert _{ x \in (-\infty,T]}$.

Then, for all $K \in [K_0 - \delta, K_0 + \delta]$, the manifold $M^-_d(K_0,x) \vert _{ x \in (-\infty,T]}$ is also strictly non-negative, since $i(x) \geq \ipinf^\ast - \frac{\rho}{2} > 0$ for all $x \in (-\infty,T]$, by Lemmas \ref{Lem:I_monotonicity} and \ref{Lem:Non_negativeness}. Moreover, at time $T$, any such manifold has entered the $\rho$-neighborhood of $(0,0,\ipinf^\ast,0)$. Thus, $i(x) \geq 0$ for all $x \in \mathds{R}$ and we can apply Lemma \ref{Lem:Non_negativeness} to conclude that also $a(x) \geq 0$ for all $x \in \mathds{R}$. Moreover, all such trajectories have at most distance $\epsilon \leq \rho$ from $M^-_d(K_0)$. Since we can choose $\rho$ as small as we want, this also proves continuity of the entire manifold $M^-_d(K,x) \vert _{ x \in \mathds{R}}$ in $K$.

\end{proof}
\end{lemma}

Continuity of the traveling waves w.r.t. $K$ implies the existence of an invading front:

\begin{theorem} \label{Thm:existence_given_nonneg}
Let $c \geq 2, r \geq 0, d > 0$ and $K_0 \in (1,2)$ such that $M^-_d(K_0)$ is a non-negative traveling wave. Then, there exists also some $K_1 \in [K_0,2)$, such that $M^-_d(K_1)$ is an invading front.
\begin{proof}
We proof a sort of intermediate value theorem, increasing $K$ as much as possible. For a manifold $M^-_d(K)$ that is non-negative and converges, we denote $\ipinf(K) := \lim_{x \rightarrow + \infty} i(x) \geq 0$.

We are done if $\ipinf(K_0) =0$, so we assume $\ipinf(K_0) >0$ in the following. Then, the manifold $K^-_d(K_0)$ fulfills the conditions of Lemma \ref{Lem:K_continuity}. It follows that there exists a neighborhood $I \subset \mathds{R}$ with $K_0 \in I$, such that the entire manifold $M^-_d(K)$ is continuous w.r.t. $K \in I$, including its limit $\ipinf(K)$. We extend this to a non-local statement by defining
\begin{align}
\begin{aligned}
K_1 := \sup_{ L \geq K_0 } \Big \{ \forall  K \in [K_0, L) \Big \vert  \,
M^-_d(K) \text{ is non-negative and converges} \Big \}.
 \end{aligned}
\end{align}
Recall that for $K=2$, the manifold $M^-_d(2)$ eventually becomes negative by Lemma \ref{Lem:I_must_become_negative}, so it must hold that $K_0 < K_1 < 2$. For all $K \in [K_0, K_1)$, the manifold $M^-_d(K)$ is non-negative and converges, by the definition of $K_1$. By Corollary \ref{Cor:Asymptotic_direction_unstable}, the manifold $M^-_d(K_1)$ exists locally around the fixed point. Since $S_d$ is a smooth vector field, also the continuation
\begin{align}
M^-_d(K_1) \vert R :=
M^-_d(K_1) \cap \big \{ x \in \mathds{R}^5: \, \vert \vert x\vert \vert _\infty < R \big \}
\end{align}
exists and is a smooth manifold, for any $R >0$. Now, since $M^-_d(K)$ is non-negative and converges for all $K \in [K_0,K_1)$, also $M^-_d(K_1) \vert R$ must be non-negative, by a continuity argument completely analogue to that in the proof of Theorem \ref{Prop:Existence_of_a_tr_sol}. Since $R$ can be chosen arbitrarily large, the entire manifold $M^-_d(K_1)$ is non-negative, and thus also converges.

We are done if $\ipinf(K_1) = 0$. So let us assume that $\ipinf(K_1) >0$. In this case, we again apply Lemma \ref{Lem:K_continuity} to $M^-_d(K_1)$, resulting in the existence of non-negative and converging solutions for in an open neighborhood of $K_1$, a contradiction.
\end{proof}
\end{theorem}

We finish the
\begin{proof}[\textbf{Proof of Theorem \ref{Thm:existence_small_epsilon}}]

For $d$ as claimed, Corollary \ref{Cor:Existence_tr_wave} and Theorem \ref{Thm:existence_given_nonneg} imply the existence of an invading front. It remains to determine its asymptotic behavior, therefore we consider only $d \in (0,1)$.

As $x \rightarrow - \infty$, the rate of convergence along the instable manifold is given in Corollary \ref{Cor:Asymptotic_direction_unstable} (depending on $\iminf$). Regarding the behavior as $x \rightarrow + \infty$, we take a look at the linear Representation \eqref{Eq:Asymptotics_linear} and the Eigenvalues \eqref{Eq:Eigenvectors_perturbed}. For $d \in (0,1)$ and $\ipinf \in [ 1 - \frac{c^2}{4},1)$, we can order the purely real-valued Eigenvalues of the limit:
\begin{align}
0>\lambda_4 \geq \lambda_3 > \lambda_2, \label{Eq:Ordering_Eigenvalues}
\end{align}
where the inequalities are strict if $\ipinf > 1 - \frac{c^2}{4}$. This ordering is crucial: even though we do not have a complete description of the asymptotics, we can determine the rate of convergence, as some components of the system converge with speed $\lambda_4$: we apply the same trapping argument as for $S_0$, see Proposition 6.3 in \cite{Kreten2022}. As long as $\lambda_4 \geq \lambda_3$, we know by an analysis of the phase space that the two components $a(x), a'(x)$ converge along the Eigenvector $e_4$ \eqref{Eq:Eigenvectors_perturbed}, associated to $\lambda_4$. We now must differentiate between the critical and the noncritical case.

For the noncritical case, where $\ipinf \neq 1 - c^2/4$, the asymptotics are equivalent to the linear System \eqref{Eq:Asymptotics_linear}. This implies that the system approaches its limit exponentially with rate $\lambda_4$.

The critical case $ \ipinf = 1 - c^2/4$ is a bifurcation point of the system: solutions that converge to $\ipinf < 1- c^2$ spiral, since the Eigenvalues $\lambda_3$ and $\lambda_4$ have an imaginary part \eqref{Eq:Eigenvectors_perturbed}, they do not spiral and converge exponentially fast if the limit fulfills $\ipinf > 1 - c^2/4$. Luckily, we do not have to find a complete representation for the behavior around the bifurcation point: we only need to determine the dynamics of a trajectory that converges to the specific limit $ \ipinf = 1 - c^2/4$. In this case, we still know that the center manifold of the fixed point $(a,a',i,i') = (0,0,\ipinf,0)$ has dimension one, and thereby must coincide with the adjacent line of fixed points. Thus, any trajectory that converges to $ \ipinf = 1 - c^2/4$ must be contained in the remaining three-dimensional strictly hyperbolic subsystem. The Eigenvalue $\lambda_3 = \lambda_4$ has algebraic multiplicity $2$, but geometric multiplicity $1$, see the Eigenvectors \eqref{Eq:Eigenvectors_perturbed}. Along the corresponding manifold, trajectories converge as fast as $x \cdot \exp( - x \lambda_4)$, cf. chapter 9 in Boyce et al. \cite{Boyce_Prima}.
\end{proof}

\appendix

\section{A-priori estimates for the left tail of the PDE} \label{Sec:A_priori_estimates}

We provide the a-priori bounds in the regime $x \rightarrow - \infty$ via a Feynman-Kac formula for shifted Brownian motion, which we stop at the origin. The following Lemma \ref{Lem:Feynman_Kac} is standard, check e.g. chapter 4.4 in \cite{Karatzas_Shreve} for more examples and a profound theoretical treatment. For linear problems (e.g. super-solutions in certain regimes of a PDE), this allows for easy estimates:

\begin{lemma} \label{Lem:Feynman_Kac}
For $c, L, M \in \mathds{R}$, let $u(t,x): \mathds{R}^+_0 \times \mathds{R} \rightarrow \mathds{R} $ be a solution of
\begin{align}
u_t = \frac{1}{2} \Delta_x u(t,x) + c u_x(t,x) + L u(t,x) + M,
\end{align}
twice differentiable in space and once differentiable in time. For fixed $x_0 < 0$, let $W_t = B_t + ct + x_0$ be a shifted Brownian motion starting in $x_0$. Denote the hitting time $T_0 := \inf \{ t \geq 0: W_t = 0\}$. Assume that for some $t >0$:
\begin{align}
\sup_{x \leq 0} \,  \vert u(t=0,x) \vert  < \infty \, \text{ and }
\sup_{s \in [0,t]} \,  \vert u(s,x=0) \vert  < \infty.
\end{align}
It then holds that
\begin{align}
\begin{aligned}
u(t,x_0) & = \mathds{E}_{x_0} \Big [
e^{L \cdot t} \cdot u \big ( 0 , W_{t}  \big ) ; \, T_0 >t
\Big ] \\
& + \mathds{E}_{x_0} \Big [
e^{L \cdot  T_0} \cdot u \big ( t- T_0 , 0  \big ); \, T_0 \leq t
\Big ] \\
& + \frac{M}{L} \cdot \mathds{E}_{x_0} \Big [ e^{L (t \wedge T_0)}  \Big ].
\end{aligned}
\end{align}
\begin{proof}
The proof is standard: fix $x_0$ and $t$ and apply Itos formula to the stochastic process
\begin{align}
X_s := e^{L \cdot s} \cdot u(t-s,W_s), \quad s \in [0,t].
\end{align}
It follows that
\begin{align}
dX_s &= L e^{L \cdot s} u(t-s,W_s) ds \nonumber \\ 
& +  e^{L \cdot s} \cdot \Big [
- u_s ds + u_x \, dBs + u_x c \, ds + \frac{1}{2} u_{xx} \, ds 
\Big ] \\
&= e^{Ls} \Big [ u_x \, dBs - M ds]. \nonumber
\end{align}
The last expression consists of a local martingale and a drift, thus
\begin{align}
u(t,x_0) = X_0 = \mathds{E}_{x_0}[X_0] = \mathds{E}_{x_0}[X_{t\wedge T_0}] + M \cdot \mathds{E}_{x_0} \Big [ \int_0^{t \wedge T_0}
e^{Ls} \, ds \Big ],
\end{align}
since $||u||_\infty$ can grow at most exponentially in time. The claim follows by splitting the cases $t < T_0$ and $t \geq T_0$.
\end{proof}
\end{lemma}

We will need the distribution of the hitting time $T_0$ as in Lemma \ref{Lem:Feynman_Kac}. It can be shown (e.g. by a Girsanov transformation and the reflection principle), that for $x_0 < 0$, the distribution of $T_0$ has density
\begin{align} \label{Eq:T0_distr}
\mathds{P}_{x_0} [ T_0 = dt] = \frac{-x_0}{\sqrt{2 \pi t^3}} \cdot e^{ - \frac{(-x_0 - ct)^2}{2t}  } \cdot \mathds{1}_{t > 0} \, dt.
\end{align}
For any speed $c>0$, this defines a probability density, so the hitting time is almost surely finite.

We can now estimate the decay of $A(t,x)$ at the back of the wave. In an analogue fashion, we could also prove that $I(t,x)$ grows at most exponentially in $x$, given that $A(t,x) \leq 1$, but we do not need this result (and in fact hope that we can also prove boundedness of $I$ in the future).

\begin{proposition}
Let $A(t,x),I(t,x)$ be a non-negative solution of the PDE \eqref{Eq:Perturbed_PDE} in the moving frame $x = z -ct$, for a speed $c > 0$. Assume that there exist constants $K, \delta, \mu_0 >0$, such that the initial data fulfill
\begin{align}
I(0, x) & \geq 1 + \delta && \qquad \forall x \leq 0, 
\nonumber \\
A(0,x) & \leq K e^{\mu_0 x} &&\qquad \forall x \leq 0, \\
A(0,x) + I(0,x) & \leq K &&\qquad \forall x \in \mathds{R}. \nonumber
\intertext{Moreover, assume that for some time $t \in (0, \infty]$, it holds that}
I(s, x = 0)  & \geq 1 + \delta > 1 && \qquad \forall s \in [0,t).
\end{align}
Then, there exist $C, \zeta >0$ that are independent of $t$, such that:
\begin{align}
\forall s \in [0,t), \, x \leq 0: \quad & i) & \quad I(s,x) \geq 1 + \delta, \label{Bound:A_pr_I_min1} \\ 
& ii) & \quad A(s,x) \leq C e^{\zeta x}. \label{Bound:A_pr_A_max1}
\end{align}

\begin{proof}
Without loss of generality let $K \geq 1$. By boundedness of the initial data, a solution exists for all times. Since $K \geq 1$, it holds that $A (t,x) \leq K$ for all $ t \geq 0$. 

For the first Bound \eqref{Bound:A_pr_I_min1}, note that the solution $H(t,x)$ of
\begin{align}
H_t = d \cdot H_{xx} + c \cdot H_z, \quad H(0,x) = I(0,x),
\end{align}
is a sub-solution for $I(t,x)$, which is not affected by any negative reaction-term. We rescale space, introducing $ x = \frac{1}{\sqrt{2d}} y$, such that $H$ fulfills $H_t = \frac{1}{2} H_{yy} + \frac{c}{\sqrt{2d}} H_y$. We can now apply Lemma \ref{Lem:Feynman_Kac} with $L = M = 0$. For all $y_0 \leq 0$ and $t \geq 0$, it holds that
\begin{align}
H(t,y_0 ) = \mathds{E}_{y_0}[H(0, W_t) ; \, T_0 >t ] + \mathds{E}_{y_0} [H(t-T_0, 0); \, T_0 \leq t] \geq 1 + \delta,
\end{align}
which is nothing but a maximum principle for a diffusion subject to a moving boundary condition.

Now, given that $I(t,x) \geq 1+\delta$ for all $x \leq 0$, the solution of
\begin{align}
J_t = J_{xx} + c \cdot J_x - \delta J
\end{align}
is a super-solution for $A$. Again, we rescale $x = \frac{1}{\sqrt{2}} y$, such that $J$ fulfills
\begin{align}
J_t = \frac{1}{2} J_{yy} + \frac{c}{\sqrt{2}} J_y - \delta J.
\end{align}
We write $\tilde{c} = c / \sqrt{2}$. It holds that for all $y_0 \leq 0$:
\begin{align}
J(t,y_0) & = \mathds{E}_{y_0} \Big [
e^{-\delta t} \cdot J \big ( 0 , W_{t}  \big ) ; \, T_0 >t
\Big ] \label{Eq:A_large_T0} \\
& + \mathds{E}_{y_0} \Big [
e^{-\delta  T_0} \cdot J \big ( t- T_0 , 0  \big ); \, T_0 \leq t
\Big ] \label{Eq:A_small_T0}.
\end{align}
We first estimate \eqref{Eq:A_large_T0}. Therefore, we fix $y_0 \leq 0$ and differentiate between small and large times. To begin with, we assume that
\begin{align}
( 2\tilde{c} + \mu_0 ) t \geq -y_0.
\end{align}
Since $A \leq K$, the term \eqref{Eq:A_large_T0} is bounded by
\begin{align}
K e^{- \delta t} \leq K e^{ \frac{\delta}{2\tilde{c} + \mu_0 } y_0  }. \label{Ineq:A_bound_1}
\end{align}
Next, we consider small times $( 2\tilde{c} + \mu_0 ) t \leq -y_0$. The term \eqref{Eq:A_large_T0} can be bounded by
\begin{align}
&\mathds{E}_{y_0} \big [
J(0,W_t) ; T_0 > t
 \big ] \leq \mathds{E}_{y_0} \big [
J(0,W_t)  \big ] \leq \mathds{E}_{y_0} \big [
K \cdot e^{ \mu_0 W_t } 
\big ]
	\nonumber \\
&= K \exp \big ( \mu_0 ( y_0 +  \tilde{c}t ) \big ) \cdot \mathds{E}_0 [ \exp (\mu_0 B_t) ],
\intertext{where $B_t$ is a standard Brownian motion. Using the moment generating function of $B_t$ and inserting our small time estimate, we get}
& = K \exp \big ( \mu_0 [ y_0 + \tilde{c}t + \frac{t \mu_0}{2}] \big) \leq K \exp ( \frac{\mu_0}{2} y_0). \label{Ineq:A_bound_2}
\end{align}
We continue with \eqref{Eq:A_small_T0}. For this, we use the distribution of the hitting time $T_0$, see \eqref{Eq:T0_distr}. In the following, let $C$ be a universal constant. We estimate
\begin{align}
\begin{aligned}
& \mathds{E}_{y_0} \Big [
e^{-\delta  T_0} \cdot J \big ( t- T_0 , 0  \big ); \, T_0 \leq t
\Big ] \leq C \mathds{E}_{y_0} [ e^{-\delta T_0 } ] \\
& = \frac{-y_0}{\sqrt{2 \pi}} \int_0 ^\infty  \exp (-\delta s) \cdot s^{-3/2} \cdot \exp \big ( -\frac{1}{2s}   (-y_0 -\tilde{c}s)^2 \big ) \, ds. 
\end{aligned}
\end{align}
Again, we differentiate between small and large times. First, consider
\begin{align}
& \frac{-y_0}{\sqrt{2 \pi}} \int_{\frac{-y_0}{2 \tilde{c} }} ^\infty  \exp (-\delta s) \cdot s^{-3/2} \cdot \exp \big ( -\frac{1}{2s}   (-y_0 -\tilde{c}s)^2 \big ) \, ds
	\nonumber \\
& \leq -Cy_0 \cdot \big( \frac{ \vert y_0 \vert }{2 \tilde{c}} \big )^{-3/2} \int_{\frac{-y_0}{2 \tilde{c}}} ^\infty  \exp (-\delta s) \, ds.
\intertext{For $y_0 \leq -1$, this can be bounded by}
& \leq
C \exp \big ( \frac{\delta}{2 \tilde{c}} y_0 \big ). \label{Ineq:A_bound_3}
\end{align}
We continue with the small time estimate:
\begin{align}
& \frac{-y_0}{\sqrt{2 \pi}} \int_0^{\frac{-y_0}{2 \tilde{c} }}  \exp (-\delta s) \cdot s^{-3/2} \cdot \exp \big ( -\frac{1}{2s}   (-y_0 -\tilde{c}s)^2 \big ) \, ds
\nonumber \\
& \leq -Cy_0  \int_0^{\frac{-y_0}{2 \tilde{c} }} s^{-3/2} \cdot \exp \big ( -\frac{y_0^2}{4s}  \big ) \, ds
\nonumber \\
& \leq -Cy_0  \frac{4}{y_0^2}  \int_0^{\frac{-y_0}{2 \tilde{c} }} \sqrt{s} \cdot \frac{y_0^2}{4s^2} \exp \big (- \frac{y_0^2}{4s}  \big ) \, ds \\
& \leq - \frac{C}{y_0} \sqrt{\frac{-y_0}{2 \tilde{c}}} \cdot \Big \vert_0^{\frac{-y_0}{2 \tilde{c} }}  \exp \big (- \frac{y_0^2}{4s}  \big ) \, ds
\nonumber \\
\intertext{For $y_0 \leq -1$, this can be bounded by}
& \leq
C \exp \big (
\frac{\tilde{c}}{2} y_0
\big ) \label{Ineq:A_bound_4}
\end{align}
Lastly, remark that for $y  \geq -1$, the trivial bound $J(t,y) \leq K$ holds, since $A \leq K$. Considering the Inequalities \eqref{Ineq:A_bound_1}, \eqref{Ineq:A_bound_2}, \eqref{Ineq:A_bound_3}, \eqref{Ineq:A_bound_4}, and re-substituting $\tilde{c} = c / \sqrt{2}$, the upper Bound \eqref{Bound:A_pr_A_max1} for $A(t,x)$ follows.
\end{proof}
\end{proposition}

\section{Numerical evaluation of the discrete spectrum}
\label{Sec:Numerical_spectrum}

We verify Assumption \ref{Ass:Point_spectrum} numerically. That is, we need to find $\delta_0, \delta_1 >0$ such that the region
\begin{align}
& \Omega = \Big \{ 
\lambda \in \mathds{C} \big \vert \, \mathfrak{Re} \lambda \geq - \delta_0 - \delta_1 \cdot  \vert \mathfrak{Im} \lambda  \vert
\Big \}, \label{Eq:Omega_Set_App}
\end{align}
contains no elements of the discrete spectrum of $\mathcal{L}$ as in Proposition \ref{Prop:Ess_Stability}. In the following, we analyze only the non-negative part of the discrete spectrum.

Given that the non-negative discrete spectrum is empty, a set as above can easily be constructed as follows. For $\mathcal{L}$ as in Prop. \ref{Prop:Ess_Stability} and for a fixed set of type \eqref{Eq:Omega_Set_App}, it is well-known that there exists a constant $K>0$ such that
\begin{align}
\forall \lambda \in \Omega \cap \Sigma_{d}: \, \vert \lambda \vert \leq K. \label{Bound_Eigenval_Abstract}
\end{align}
This can be proven by rescaling the Eigenvalue-ODE \eqref{Def:Operator_as_ODE} for large $\lambda$, see e.g. Prop. 2.2 in \cite{Alexander_Gardner_Evans_Invariance}, or Section 2.2 in \cite{Gardner_Zumbrun_Gap_Lemma}. Since the Evans function (see Section \ref{Sec:spectrum_evans}) is menomorphic over $\Omega$ away from the continuous spectrum (Prop. 4.7 in \cite{Zumbrun_Howard_Pointwise_Stability}), the bounded set
\begin{align}
\Omega_K^- := \Omega \cap  \{ \lambda \in \mathds{C} \big \vert \, \mathfrak{Re} \lambda < 0 , \vert \lambda \vert \leq K \}
\end{align}
can contain at most finitely many elements of the discrete spectrum. Thus, if the non-negative part of $\Omega$ contains no element of the discrete spectrum, one finds a set of type \eqref{Eq:Omega_Set_App} as claimed by eventually choosing $\delta_0, \delta_1$ smaller.

Now, for the numerical analysis of the discrete spectrum, we consider only a single type of weight, namely
\begin{align}
w(x) = e^{-x}. \label{Eq:Weight_Appendix_Numeric}
\end{align}
This choice allows for an explicit and rather efficient energy based estimate regarding the maximal possible size of an Eigenvalue with non-negative real-part, presented in Section \ref{Sec:Energy_bound}. This estimate is much better than the asymptotic Bound \eqref{Bound_Eigenval_Abstract}, where an explicit calculation of the constant $K$ is a rather tedious task. However, this approach is restricted to fully diffusive systems. In the subsequent Section \ref{Sec:spectrum_evans}, we briefly review the theory of the Evans function and how it can be used to analyze the discrete spectrum. For the numerical evaluation, we use STABLAB, a tool by Barker \textit{et al.} \cite{Barker_Stablab}. These computations show that $\mathcal{L}$ with the Weight \eqref{Eq:Weight_Appendix_Numeric} is indeed spectrally stable in $H^2$.

If the non-negative discrete spectrum for the Weight \eqref{Eq:Weight_Appendix_Numeric} is empty, then the same is obviously also true for all
\begin{align}
w(x)_{\vert x \leq -1} = e^{-\alpha_-x}, \quad \alpha_- < 1.
\end{align}

\subsection{An energy bound}
\label{Sec:Energy_bound}
The below idea is probably due to J. Humpherys, who also presents a refined version in \cite{Hendricks_Humpherys_Energy_Bound_Stability}. A similar bound is found by F. Varas and J. Vega \cite{Varas_Vega_Wrong_bound}, but their calculation turns out to be wrong due to a sign error. Consider the linear problem
\begin{align}
\lambda \cdot u_i(x) = D_i \cdot u_i''(x) + c_i \cdot u_i'(x) + \sum_{j=1}^n f_{i,j}(x) \cdot u_j(x), && i = 1, \dotso, n, \label{Eq:Linear_problem}
\end{align}
for $\lambda \in \mathds{C}$, constants $D_i \geq 0, c_i \geq 0$, and functions $f_{i,j} \in L^\infty(\mathds{R})$.

Now assume that there exists an Eigenvalue $\lambda \in \mathds{C}$ with $\mathfrak{Re} \gamma \geq 0$, and associated eigenfunctions $u_i \in H^2(\mathds{R})$. The product structure of $H^2$ induces an energy estimate on $ \vert \lambda \vert $:

\begin{theorem}[An a-priori bound for the positive discrete spectrum] \label{Thm:Spectral_bound} 
Let $\lambda \in \mathds{C}$ with $\mathfrak{Re} \lambda \geq 0$ and functions $\big ( u_i \in H^2(\mathds{R}) \big )_{i = 1 \dotso n}$ be a solution of the Eigenvalue-problem \eqref{Eq:Linear_problem}. For $i = 1, \dots, n$, define
\begin{align}
M_i := \sum_{j=1}^n  \vert \vert f_{i,j} \vert \vert _\infty. \label{Eq:Mi_eigenval_constants}
\end{align}
Then, the real-part of $\lambda$ is bounded:
\begin{align}
\mathfrak{Re} \lambda & \leq  \max_i \{ M_i \}. \label{Eq:Real_part_bounded}
\intertext{Moreover, if $D_i>0$ for all $i = 1, \dots, n$, such that the system is with non-degenerate diffusion, it holds that}
 \vert \mathfrak{Im} \lambda  \vert  & \leq \max_i \Big \{ c_i \cdot \sqrt { \frac{M_i}{D_i}  } + M_i \Big \} .
\end{align}
\begin{proof}
In the following, the norm $ \vert  \vert . \vert  \vert_2 $ will be the $L^2$-norm with corresponding scalar product $\langle \, . \, \rangle$. Before proving the above statement, let us note that for any function $u \in H^2(\mathds{R})$, the product $\langle u', \bar{u} \rangle$ has no real-part, since
\begin{align}
\begin{aligned}
\mathfrak{Re} \, \int_{\mathds{R}} u'(x) \cdot \bar{u}(x) \, dx & = \frac{1}{2} \Big ( \int_{\mathds{R}} u'(x) \cdot \bar{u}(x) + \bar{u}' \cdot u(x) \, dz \Big ) \\& = \frac{1}{2} \int_{\mathds{R}}  \vert u(x)^2 \vert ' \, dx = 0.
\end{aligned}
\end{align}
Moreover, it holds by integration by parts that
\begin{align}
\langle u'', \bar{u} \rangle = -  \vert  \vert u' \vert  \vert_2 ^2 \leq 0.
\end{align}

Now, assume that there exists an Eigenvalue $\lambda \in \mathds{C}$ with $\mathfrak{Re} \lambda \geq 0$ and associated eigenfunctions $\big ( u_i \in H^2(\mathds{R}) \big )_{i = 1 \dotso n}$ that solve \eqref{Eq:Linear_problem}. For for each $i \in 1 \dotso n$, multiply \eqref{Eq:Linear_problem} by $\bar{u}_i$:
\begin{align}
	\lambda \cdot  \vert  \vert u_i \vert  \vert_2 ^2 &= D_i \cdot \langle u_i'', \bar{u_i} \rangle + c_i \cdot \langle u_i', \bar{u_i} \rangle + \sum_{j=1}^n \langle f_{i,j} \cdot u_j, \bar{u_i} \rangle \nl
	&= - D_i  \vert  \vert u_i' \vert  \vert_2 ^2 + c_i \cdot \langle u_i', \bar{u_i} \rangle + \sum_{j=1}^n \langle f_{i,j} \cdot u_j, \bar{u_i} \rangle \label{Eq:Quadratic_form_lambda_ui} .
\end{align}
Choose $k \in 1 \dotso n$ such that $ \vert  \vert u_k \vert  \vert_2  = \max_i  \vert  \vert u_i \vert  \vert_2 $. By Eq. \eqref{Eq:Quadratic_form_lambda_ui}, and since $\langle u_i', \bar{u_i} \rangle$ has no real-part:
\begin{align}
\begin{aligned}
	0 \leq \mathfrak{Re} \lambda \cdot  \vert  \vert u_k \vert  \vert_2 ^2 &= - D_k \cdot  \vert  \vert u_k' \vert  \vert_2 ^2 + \sum_{j=1}^n \mathfrak{Re}  \langle f_{k,j} \cdot u_j, \bar{u_k} \rangle \\
	& \leq  \vert  \vert u_k \vert  \vert_2 ^2 \cdot  \sum_{j=1}^n  \vert \vert f_{k,j} \vert \vert _\infty.
\end{aligned}
\end{align}
The claimed bound for $\mathfrak{Re} \lambda$ follows. We proceed similarly for the imaginary part of Eq. \eqref{Eq:Quadratic_form_lambda_ui}:
\begin{align}
\mathfrak{Im} \lambda \cdot  \vert  \vert u_i \vert  \vert_2 ^2 & = c_i \cdot \langle u_i',\bar{u}_i \rangle  + \sum_{j=1}^n \mathfrak{Im} \langle f_{i,j} \cdot u_j, \bar{u}_i \rangle .
\intertext{Again, for $ \vert  \vert u_k \vert  \vert_2  = \max_i  \vert  \vert u_i \vert  \vert_2 $:}
 \vert \mathfrak{Im}\lambda \vert  \cdot  \vert  \vert u_k \vert  \vert_2 ^2 &\leq c_k \cdot  \vert  \langle u_k',\bar{u}_k \rangle   \vert  + \sum_{j=1}^n \vert \vert f_{k,j} \vert \vert _\infty \cdot  \vert  \vert u_k \vert  \vert_2 ^2. \label{Eq:Im_part_u_prime_missing}
\end{align}
We are done if we can find a suitable bound for $ \vert  \vert u'_k \vert  \vert_2 $. This is where we require that $D_i>0$ for all $i \in 1 \dotso n$. Again take the real-part in Eq. \eqref{Eq:Quadratic_form_lambda_ui}. The inequality $\mathfrak{Re} \lambda \geq 0$ implies that for all $i=1, \dots, n$:
\begin{align}
D_i \cdot  \vert  \vert u_i' \vert  \vert_2 ^2 & \leq \sum_{j=1}^n \vert  \vert f_{i,j} \vert \vert _\infty \cdot  \vert  \vert u_i \vert  \vert_2 \cdot  \vert  \vert u_j \vert  \vert_2 ,
\end{align}
and thus for $k$ chosen as above: $ \vert  \vert u_k' \vert  \vert_2  \leq \sqrt{ \frac{M_k}{D_k} } \cdot  \vert  \vert u_k \vert  \vert_2 $. We use this in Eq. \eqref{Eq:Im_part_u_prime_missing} to bound
\begin{align}
 \vert \mathfrak{Im} \lambda \vert  \leq c \cdot \sqrt { \frac{M_k}{D_k} } +  M_k. \label{Eq:Im_bound_case_1}
\end{align}
\end{proof}
\end{theorem}

We apply Theorem \ref{Thm:Spectral_bound} to the Operator \eqref{Eq:Linear_Operator} that already encodes the transformation into the weighted space. For a traveling wave $a(x), i(x)$ with speed $c=2$ and the weight $w(x) = e^{-x}$, $\mathcal{L}$ reduces to
\begin{align}
\begin{aligned}
\mathcal{L}u &:= \frac{\partial^2}{ \partial x^2} u - u  (2a+i  ) -va,  \\
\mathcal{L}v &:= d \frac{\partial^2}{ \partial x^2} v  + ( 2 - 2d  ) \frac{\partial}{ \partial x} v  + v \cdot ( a + d -2  ) + u(  2a+i+r). \label{EQ:Reduced_Lin_Op}
\end{aligned}
\end{align}
Here, it is essential that $w'/w=-1$ is constant, such that we are in the setting of Theorem \ref{Thm:Spectral_bound}.

\begin{proposition} \label{Prop:Eigenvalue_bound}
For $d \in (0,1)$, consider $\mathcal{L}$ \eqref{EQ:Reduced_Lin_Op} as an operator
\begin{align}
H^2(\mathds{R}) \times H^2(\mathds{R}) \rightarrow  L^2(\mathds{R}) \times L^2(\mathds{R}).
\end{align}
For all $\lambda \in \mathds{C}$ in the discrete spectrum of $\mathcal{L}$ with $\mathfrak{Re}\lambda \geq 0$, it holds that
\begin{align}
 \vert \lambda \vert  & \leq \sqrt{2}  \Big [ (2-2d) \sqrt{ \frac{5+r}{d} } + 5+r \Big ].
\label{Eq:Practical_Energy_Bound}
\end{align}
\begin{proof}
This is a direct consequence of Theorem \ref{Thm:Spectral_bound}. We can estimate the two constants $M_a$ and $M_i$, see Eq. \eqref{Eq:Mi_eigenval_constants}, via the simple bounds $ \vert i \vert  \leq 2,  \vert a \vert  \leq 1$:
\begin{align}
M_a  &=  \vert 2a-i \vert _\infty +  \vert a \vert _\infty  \leq 3, \\
M_i &=  \vert a+d-2 \vert _\infty +  \vert 2a+i+r \vert _\infty \leq 5+r .
\end{align}
\end{proof}
\end{proposition}

\subsection{The Evans function}
\label{Sec:spectrum_evans}

The discrete spectrum within the region of consistent splitting of an operator can be evaluated numerically with the help of the Evans function. The following concepts are presented in much more detail by B. Sandstede, who wrote a great overview of the topic \cite{SANDSTEDE_stability_traveling}, a more practical introduction and the numerical tool STABLAB \cite{Barker_Stablab} have been written by B. Barker et al. \cite{Barker_Evans_computations}.

Recall that for $\mathcal{L}$ as in Proposition \ref{Prop:Ess_Stability}, if $\lambda \in \Sigma_d$, this is equivalent to saying that there exists a bounded solution of a linear ODE of type
\begin{align}
U' = M(x,\lambda) \cdot U, \qquad x \in \mathds{R}, \label{Def:Operator_as_ODE_repeated}
\end{align}
compare \eqref{Def:Operator_as_ODE}. The matrix $M(x, \lambda)$ is given by \eqref{Def:Matrix_general_case} and encodes the shift of the equation into the weighted space. It is now essential that we consider only $\lambda \in \mathds{C}$ such that the matrices $M(x, \lambda)$ are analytic in $(\lambda,x)$ up to $x = \pm \infty$, with strictly hyperbolic limits $M(\pm \infty, \lambda)$.

Since the limiting matrices $M(\pm \infty, \lambda)$ are hyperbolic, the theory of \textit{exponential dichotomies} implies that any bounded solution $U(x)$ must vanish exponentially fast as $x \rightarrow \pm \infty$. Moreover, it asymptotically approaches the unstable manifold of the constant matrix $M(- \infty, \lambda)$ as $x \rightarrow - \infty$, and the stable manifold of $M(+ \infty, \lambda)$ as $x \rightarrow + \infty$. Therefore, a bounded solution exists if and only if the trajectories that emerge from these manifolds intersect. This allows us to compute the Evans function: it is a determinant that evaluates to zero if and only if the solutions that decay at $-\infty$ and those that decay at $+\infty$ intersect, and are linearly dependent. The details are given in Section 4.1 in \cite{SANDSTEDE_stability_traveling}.

\begin{definition}[Evans function] \label{Def:Evans_function}
Consider the ODE \eqref{Def:Operator_as_ODE_repeated}, with $\lambda \in \mathds{C}$ such that the matrices $M(x, \lambda)$ are analytic in $(\lambda,x)$ up to $x = \pm \infty$, with hyperbolic limits $M(\pm \infty, \lambda)$.

Let
\begin{align}
&X_1(x, \lambda), \dotso, X_{k_1}(x, \lambda)
\intertext{be $k_1$ linearly independent representatives of those solutions that decay exponentially as $x \rightarrow - \infty$, and let}
&Y_1(x,\lambda), \dotso, Y_{k_2}(x,\lambda)
\end{align}
be $k_2$ linearly independent representatives of those solutions that decay exponentially as $x \rightarrow + \infty$, with $k_1, k_2 >0$ and $k_1 +  k_2 = n$. The Evans function $E(\lambda)$ is defined as
\begin{align}
E( \lambda ) := \text{det} \big ( X_1(x, \lambda) \big \vert \dotso  \big \vert X_{k_1}(x,\lambda) \big \vert Y_1(x, \lambda)  \big \vert \dotso Y_{k_2, \lambda} \big )  \Big \vert_{x=0}. \label{Eq:Evans}
\end{align}
\end{definition}

\textit{Remark}: The Evans function is not unique, since the representatives $X_i, Y_i$ are not unique. However, it holds that $E(\lambda) = 0$ if and only if System \eqref{Def:Operator_as_ODE_repeated} has a bounded solution. Moreover, $E(\lambda)$ is analytic if the $X_i, Y_i$ are chosen analytically in $\lambda$, which can be achieved, see section 4.1 in \cite{SANDSTEDE_stability_traveling}. Then, it suffices to calculate $E(\lambda)$ along the boundary of a domain $\mathcal{P} \subset \mathds{C}$: the winding number of $E(\lambda)$ along $\partial \mathcal{P}$ corresponds to the number of zeros inside the domain.\\

We first need a numerical solution of a wave-profile, which goes hand in hand with finding the correct value of $\iminf$ such that the the other limit of the traveling wave fulfills $\ipinf = 0$. Given some $\iminf  >1 $, we solve the forward ODE \eqref{Eq:perturbed_wave_ODE} under initial data that approximate the unstable manifold of a fixed point $(a,a',i,i') = (0,0,\iminf,0)$, its asymptotic direction is given in Corollary \ref{Cor:Asymptotic_direction_unstable}. The limit as $x \rightarrow + \infty$ of this trajectory is Lyapunov-stable, such that the numerical forward solution converges if the initial data are sufficiently close to the unstable manifold of $ (0,0,\iminf,0)$. We then change $\iminf$ until $ \ipinf=0$. This differs from the approach suggested in \cite{Barker_Evans_computations}: the continuum of fixed points makes the proposed projective boundary value approach invalid. However, since the trajectory converges along the correct stable manifold as $x \rightarrow + \infty$, our forward approach is reasonable. This is a consequence of the trapping argument in Proposition \ref{Triangles_3D_Thm}, see also Proposition 6.3 in \cite{Kreten2022}. The numerical evaluation suggests that the value $\ipinf$ is a monotone function of $\iminf$, such that any root-finding algorithm converges quickly to the (presumably) unique value of $\ipinf$ such that $\iminf = 0$. The results are depicted in Figure \ref{Tab:Root_iminf}.

\begin{figure}[h]
\begin{tabular}{l | c | c | c | c | c | c}
\multicolumn{7}{c}{}    
\\ 

\multicolumn{2}{r | }{  $d = 0.001 $ } & $d = 0.01 $ & $d = 0.1 $& $d = 0.2 $ & $d = 0.3 $ & $d = 0.4 $ \\ \hline
     $r=0$ & $1.99985$ & $1.99848$ & $1.98489$ & $1.96999$ & $1.95532$ & $1.94091$ \\ \hline
     $r=1$ & $1.99984$ & $1.99842$ & $1.98430$ & $1.96897$ & $1.95403$ & $1.93948$ 
\end{tabular}
 \caption{The value of $\lim_{x \rightarrow - \infty}i(x)$ of the invading fronts (where $\ipinf = 0$), for different values of $r, d$ and for speed $c=2$. The numerical values suggest that the effect of $d$ is negligible, and that our bound in Lemma \ref{Lem:Epsilon_limit_bounds} is not sharp. The root finding was performed with Wolfram Mathematica.} \label{Tab:Root_iminf}
\end{figure}

\begin{figure}[h]
\centering
\includegraphics[width = 0.5\textwidth]{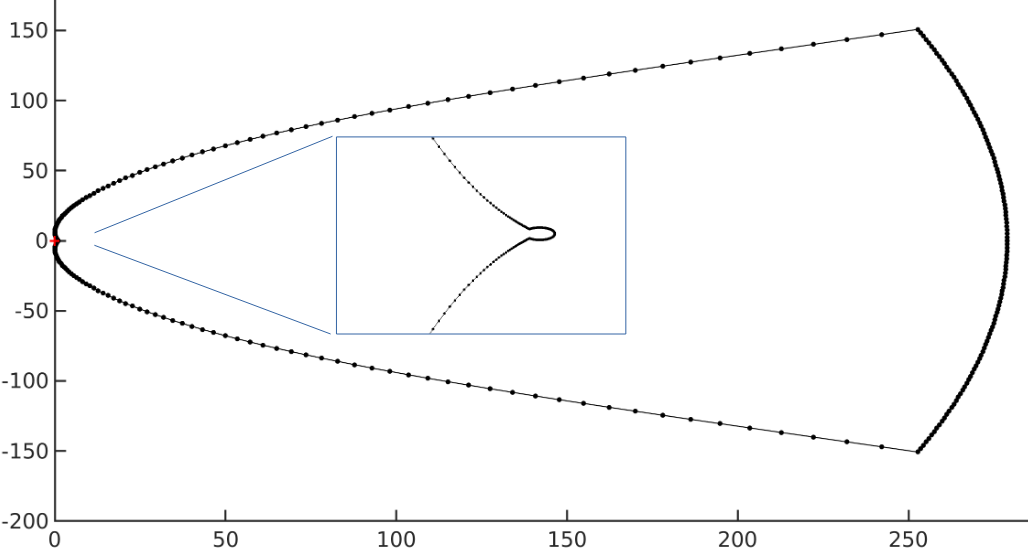}
\caption{The Evans function for the weight $w(x) = e^{-x}$, for $c=2.0, d = 0.3$ and $r=1$, evaluated along the boundary of the domain $\mathcal{P}$, see \eqref{Eq:Domain_Evans}. The boundary consists of three parts: a semisircle with radius $\delta = 10^{-3}$ around the origin, two lines along the imaginary axis, and a semi-circle with radius $R$, see Proposition \ref{Prop:Eigenvalue_bound}. The winding number is equal to zero, as a consequence there are no elements of the discrete spectrum of $\mathcal{L}$ within $\mathcal{P}$. The simulations for all other parameters result in similar pictures.} \label{Fig:Evans}
\end{figure}

The various numerical challenges that arise when computing the Evans function, as well as their solutions, are described in detail by Barker et al. \cite{Barker_Evans_computations}, who also suggest using their library STABLAB \cite{Barker_Stablab}. Given a numerical solution of the traveling front on an interval $x \in [-L,+L]$ (we choose $L = 50$), centered such that $a'(0) = 0$, STABLAB performs three main steps: first, the unstable and stable spaces of System \eqref{Def:Operator_as_ODE_repeated} are approximated by the unstable and stable eigenspaces of the constant matrices $M(\pm \infty, \lambda)$. Second, bases of these eigenspaces that are analytic in $\lambda$ are chosen. Third, the resulting ODEs are solved (forwards on $[-L,0]$ and backwards on $[0,L]$), in an exponentially weighted space to reduce computation time and increase numerical stability. We evaluate the Evans function along $\partial \mathcal{P}$, the boundary of the domain
\begin{align}
\mathcal{P} := \{ \gamma \in \mathds{C}\big \vert \, \mathfrak{Re} \gamma \geq 0, \, \delta \leq  \vert \gamma \vert  \leq R \}, \label{Eq:Domain_Evans}
\end{align}
where the radius $R$ is the upper bound for any element of the discrete spectrum with non-negative real-part from Proposition \ref{Prop:Eigenvalue_bound}, and a small distance $\delta = 10^{-3}$, since the essential spectrum touches the origin. Note that analytically, the Evans function can be extended up to the origin, see e.g. \cite{FAYE_lotke_volterra_critical_stability}, due to the lack of a discrete eigenvalue zero.

As $d \rightarrow 0$, the computation of the Evans function gets numerically unstable. We get reliable results for $d \geq 0.1$, and found no points of the non-negative discrete spectrum for all choices of $r,d$ as in Table \ref{Tab:Root_iminf}. One of the resulting contours is depicted in Figure \ref{Fig:Evans}, all others were very similar.

\section{Geometric singular perturbation theory}
\label{Sec:Geom_sing_appl}

For passing continuously from $d = 0$ to $d \ll 1$ in System \eqref{Eq:perturbed_wave}, we apply geometric singular perturbation theory for smooth systems, based on work by F\'{e}nichel \cite{Fenichel_orgiginal_work, Jones_Singular_Perturbation}. We present the underlying theory and its application to the present case (briefly, since it is rather standard). Given two $C^\infty$ functions $f$ and $g$, consider the following smooth field:
\begin{align}
\begin{aligned}
\frac{d }{dt} z &= f(z,y,d), \\
\frac{d }{ dt} y &= \epsilon g(z,y,d), \quad \epsilon \sim 0. \label{Eq:Fenichel_fast_new}
\end{aligned}
\end{align}
The variable $z$ is called the \textit{fast variable}, the variable $y$ is called the \textit{slow variable}. Geometric singular perturbation theory leads to a result that describes trajectories that are continuous in $\epsilon$ in a neighborhood of $\epsilon = 0$. Introduced by F\'{e}nichel in the 70s \cite{Fenichel_orgiginal_work}, this approach has been classical for dealing with singularly perturbed systems, check \cite{Jones_Singular_Perturbation}. The central assumption is given as follows, it ensures that the fast variables indeed move along fast stable and unstable manifolds:
\begin{definition}
For $z \in \mathds{R}^n, y \in \mathds{R}^l$, consider a fixed point $P \in \mathds{R}^{n+l}$ of a dynamical system as in \eqref{Eq:Fenichel_fast_new}. It is called a \textit{normally hyperbolic} critical point, if the linearization of the system around $P$ has exactly $l$ Eigenvalues that lie on the imaginary axis.
\end{definition}
For $\epsilon = 0$, the $l$ Eigenvalues with zero real-part of the Jacobian at a normally hyperbolic point must correspond to the slow (constant for $\epsilon = 0$) variables. Only the movement in the fast direction remains. In this setting, a separation of time-scales occurs for $\epsilon \neq 0$ sufficiently small, and one can treat the fast and slow variables separately. The following theorem is a combination of Theorems 1-3 in the monograph of C. Jones \cite{Jones_Singular_Perturbation}, this compact form is due to Rottschäfer \cite{ROTTSCHAFER_Sing_Pert_Wave}:

\begin{theorem} \label{Thm:Fenichel_new}
Given a system of type \eqref{Eq:Fenichel_fast_new}, such that when $\epsilon = 0$, the system of equations has a compact, normally hyperbolic manifold of critical points, $\mathcal{M}_0$, where we assume that $\mathcal{M}_0$ is given as the graph of a $C^\infty$ function $h^0(y)$. Then for every $p \geq 1$, there exists an $\epsilon_0>0$ such that for all $\epsilon \in [0, \epsilon_0)$, there exists a Manifold $\mathcal{M}_\epsilon$:
\begin{enumerate}
\item $\mathcal{M}_\epsilon$ is locally invariant under the flow defined by \eqref{Eq:Fenichel_fast_new}.
\item $\mathcal{M}_\epsilon = \{ (z,y) \vert z = h^\epsilon(y) \}$, for a function $h^\epsilon$ which is $C^p$ in both $\epsilon$ and $y$, and for $y$ in some compact set $S$.
\item There are locally stable manifolds $W^s(\mathcal{M}_\epsilon)$ and unstable ones $W^u(\mathcal{M}_\epsilon)$, that lie within $\mathcal{O}(\epsilon)$ of and are diffeomorphic to $W^s(\mathcal{M}_0)$ and $W^u(\mathcal{M}_0)$, respectively. Here, $W^s(\mathcal{M}_0)$ and $W^u(\mathcal{M}_0)$ are the (fast) stable and unstable manifolds of the unperturbed manifold $\mathcal{M}_0$. 
\end{enumerate}
\end{theorem}

We follow the examples in \cite{Jones_Singular_Perturbation, ROTTSCHAFER_Sing_Pert_Wave}, and analyze the present system for $d \sim 0$. We bring System \eqref{Eq:perturbed_wave_ODE} into slow-fast form as in Eq. \eqref{Eq:Fenichel_fast_new}, and change the phase-time to the fast phase-time $t := x / d$:
\begin{align}
\begin{aligned} \label{Eq:System_Fenichel_Wave_new}
\frac{d}{ d t } i' &= - [ci' + ra + a(a+i)] =: f(a,a',i,i'), \\
\frac{d}{ d t } \begin{pmatrix}
a \\ a' \\ i
\end{pmatrix}
&= d \cdot \begin{pmatrix}
a' \\
a(a+i) -a -ca'\\ i'
\end{pmatrix} =: d \cdot g(a,a',i,i').
\end{aligned}
\end{align}
Here $(a,a',i) $ are the slow variables and $i' $ is the fast variable. For $d = 0$, there exists a three-dimensional smooth manifold of critical points
\begin{align}
\mathcal{M}_0 = \Big \{ (a,a',i,i') \in \mathds{R}^4\Big \vert \, i' = h^0(a,a',i) :=  - \frac{ a(a+i+r) }{c }
\Big \}. \label{Eq:Fenichel_M0_system}
\end{align}
We check normal hyperbolicity of $\mathcal{M}_0$. For $d = 0$, we calculate the Jacobian of the system:
\begin{align}
J_{(a,a',i,i') \vert M_0} = \begin{pmatrix}
0 & 0 &  0 & 0 \\
0 & 0 &  0 & 0 \\
0 & 0 &  0 & 0 \\
r + 2a +i & 0 & a & -c
\end{pmatrix},
\end{align}
which has a single non-zero Eigenvalue $\lambda = -c$, with fast direction $i'$. Thus, all points in $\mathcal{M}_0$ have a fast stable manifold of dimension one and no unstable manifold (in fact, $\mathcal{M}_0$ is even a global attractor). In order to apply Theorem \ref{Thm:Fenichel_new}, we now choose an arbitrarily large, but bounded subset of $\mathcal{M}_0$. Then, for $d >0$ sufficiently small, there exists a manifold $\mathcal{M}_d$, that is of distance $\mathcal{O}(d)$ to $\mathcal{M}_0$, together with a local unstable manifold $W^s(\mathcal{M}_d)$ that is of distance $\mathcal{O}(d)$ to $W^s(\mathcal{M}_0)$.

For $d \neq 0$, we can undo our transformation and change back to $x= d \cdot t$. Then, we get that on $\mathcal{M}_d$, the perturbed system is close unperturbed system, in the sense that
\begin{align}
\frac{d}{dx} \begin{pmatrix}
a \\ a' \\ i
\end{pmatrix}
& =\begin{pmatrix}
a' \\
a(a+i) -a -ca'\\ i'
\end{pmatrix}
\intertext{with $i'=h^d(a,a',i)$, subject to the condition}
ci' & =  - a(a+i+r) + \mathcal{O}(d), \label{Eq:Eps_closeness_fenichel}
\end{align}
which is an $\mathcal{O}(d)$-perturbation of \eqref{Eq:Fenichel_M0_system}. Equation \eqref{Eq:Eps_closeness_fenichel} implies that the restriction of $S_d$ to $\mathcal{M}_d$ is $\mathcal{O}(d)$-close to the unperturbed system $S_0$. Thus, along $\mathcal{M}_d$, the System $S_d$ can be viewed as a regular perturbation of $S_0$, which we obtain as a (regular) limit as $d \rightarrow 0$:
\begin{align}
\frac{d}{dx} \begin{pmatrix}
a \\ a' \\ i
\end{pmatrix}
& =\begin{pmatrix}
a' \\
a(a+i) -a -ca'\\ - \frac{a(a+i+r)}{c}
\end{pmatrix}. \label{Eq:Unperturbed_fenichel_new}
\end{align}
We summarize this in

\begin{corollary} \label{Cor:Feniches_our_system}
Fix $p \geq 1$ and choose any smooth bounded set $\bar{\mathcal{M}_0} \subset \mathcal{M}_0$, for $\mathcal{M}_0$ as defined in \eqref{Eq:Fenichel_M0_system}. There exists some $d^\ast >0$, such that for all $d \in (0, d^\ast)$, the System \eqref{Eq:System_Fenichel_Wave_new} has a three-dimensional $C^p$-manifold $\bar{\mathcal{M}_d}$, invariant under $S_d$ and $\mathcal{O}(d)$-close to $\bar{\mathcal{M}_0}$. The flow on this manifold is an $\mathcal{O}(d)$-perturbation of \eqref{Eq:Unperturbed_fenichel_new}. Moreover, $\bar{\mathcal{M}_d}$ has no unstable manifold, but only a fast stable manifold $W^s(\bar{\mathcal{M}_d})$ that is locally diffeomorphic to and within range of $\mathcal{O}(d)$ to $W^s(\bar{\mathcal{M}_0})$.
\end{corollary}

Now let $K >1$ and consider $d>0$ sufficiently small. The fixed point $(a,a',i,i') = (0,0,K,0)$ has an unstable manifold of dimension one (presumably a traveling wave). With the help Corollary \ref{Cor:Feniches_our_system}, we can track any finite segment of this manifold:

\begin{corollary}[c.f. Corollary \ref{Cor:Cont_finite}]
Let $K>1, c >0, r \geq 0$. First consider the fixed point $(\bar{a},\bar{a}',\bar{i}) = (0,0,K)$ of $S_0$, together with its one-dimensional unstable manifold $M^-_0(K)$. Fix any semi-open interval $ x \in (-\infty, T]$, where $T$ is finite, and assume that $M^-_0(K,x)\vert_{x \in (-\infty, T]}$ is smooth and bounded. Lift it into $\mathds{R}^4$ via \eqref{Eq:Fenichel_M0_system}.

Now consider the perturbed system $S_d$. There exist some $d^\ast > 0$ such that for all $d \in (0,d^\ast)$: the fixed point $(a,a',i,i') = (0,0,K,0)$ has an adjacent one-dimensional unstable manifold $M^-_d(K,x)\vert_{x \in (-\infty, T]}$, that is continuous in $d$ and converges to $M^-_0(K,x)\vert_{x \in (-\infty, T]}$ as $d \rightarrow 0$.

\begin{proof}
Fix a finite time-horizon $T \in \mathds{R}$ such that $M^-_0(K,x)\vert_{x \in (-\infty,T]}$ is smooth and bounded. Embed it into $\mathds{R}^4$ by setting $i'=h^0(a,a',i)$, see \eqref{Eq:Fenichel_M0_system}. Let $\bar{M_0}$ be a smooth bounded subset of $M_0$, sufficiently large such that
\begin{align}
M^-_0(K,x)\vert_{x \in (-\infty,T]} \subset \bar{M_0}.
\end{align}
There exists an $d^\ast>0$, such that for all $d \in [0, d^\ast)$: there exists a $C^p$ manifold $\bar{M_d}$ that is invariant under $S_d$, and the restriction of $S_d$ to $M_d$ is a perturbation of $S_0$ \eqref{Eq:Unperturbed_fenichel_new}.

Since $(a,a',i,i') = (0,0,K,0)$ is a fixed point, it can not lie on $W^s(\bar{M_d})$, and therefore must lie within $\bar{M_d}$ itself. Similarly, its unstable manifold, given in Corollary \ref{Cor:Asymptotic_direction_unstable}, must lie within $\bar{M_d}$. However, along $\bar{M_d}$, the flow $S_d$ is a regular perturbation of the unperturbed flow $S_0$, with a perturbation of order $\mathcal{O}(d)$. Thus, the finite segment $M^-_0(K,x)\vert_{x \in (-\infty,T]}$ is continuous in $d$, for $d \in (0, d_1)$, with $d_1$ possibly smaller than $d^\ast$, by a regular perturbation analogue to Proposition \ref{Prop:Epsilon_continuity}.
\end{proof}
\end{corollary}

\textbf{Numerical analysis:} The numerical analysis of the spectrum was performed via STABLAB \cite{Barker_Stablab}. The simulations of the Reaction-Diffusion System \eqref{Eq:Perturbed_PDE} were done with Wolfram Mathematica. The code can be accessed upon request.

\bibliography{bib_reduced.bib}{}
\bibliographystyle{plain}

\end{document}